\documentclass[reqno,10pt]{article}

\usepackage{xcolor}
\usepackage{graphicx} 
\usepackage{caption}
\usepackage{subcaption}
\usepackage{float}
\usepackage{amsmath}
\usepackage{amssymb}
\usepackage{amsthm}
\usepackage{enumitem}
\usepackage{bbm}
\usepackage[l2tabu,orthodox]{nag}
\usepackage[strict]{csquotes}
\usepackage{url}

\newtheorem{theorem}{Theorem}
\newtheorem{lemma}[theorem]{Lemma}
\newtheorem{corollary}[theorem]{Corollary}
\theoremstyle{definition}

\newtheorem{remark}[theorem]{Remark}

\DeclareMathOperator{\sign}{sign}
\newcommand{\norm}[1]{\left\lVert#1\right\rVert}
\newcommand{\set}[1]{\left\{#1\right\}}

\numberwithin{equation}{section}
\numberwithin{theorem}{section}

\newenvironment{OMabstract}{\noindent\textbf{Abstract.} }{\medskip}
\newenvironment{OMsubjclass}{\noindent\textbf{Mathematics Subject Classification (2020):} }{\medskip}
\newenvironment{OMkeywords}{\noindent\textbf{Keywords:}  }{\medskip}

\begin{document}

\author{Filipp Buryak, Yuliya Mishura}
\title{Convexity and robustness of the R\'{e}nyi entropy}
\maketitle

\begin{OMabstract}
We study convexity properties of R\'{e}nyi entropy as function of $\alpha>0$ on finite alphabets. We also describe  robustness of the R\'{e}nyi entropy on finite alphabets, and it turns out that the rate of respective convergence depends on initial alphabet. We  
establish convergence of the disturbed entropy when the initial distribution is uniform   but the number of events increases to $\infty$ and  prove that limit of R\'{e}nyi entropy of binomial distribution is equal to R\'{e}nyi entropy of Poisson distribution.
\end{OMabstract}

\begin{OMkeywords}
Discrete distribution, R\'{e}nyi entropy, Convexity.
\end{OMkeywords}

\begin{OMsubjclass}
60E05, 94A17.
\end{OMsubjclass}

\section{Introduction}
Let $(\Omega, \mathfrak{F}, \mathbf{P})$ be a probability space supporting all distributions considered below.  For any $N\ge 1$ introduce the family of discrete distributions  $p=(p_1,p_2,\ldots,p_N)$   with probabilities  $$p_i\geq0, 1\le i\le N, N\ge 1,\;  p_1+...+p_N=1.$$ In the present paper we  investigate some properties  of the R\'{e}nyi entropy, which was proposed by R\'{e}nyi in \cite{Renyi},
\begin{align*}
\mathcal{H}_\alpha(p) = \frac{1}{1-\alpha}\log\left( \sum_{k=1}^{N} p_k^{\alpha}\right), \; \alpha > 0, \; \alpha \neq 1,
\end{align*}
including its limit value as $\alpha\rightarrow 1$, i.e.,  the Shannon entropy
\begin{align*}
\mathcal{H}(p)=-\sum_{k=1}^{N} p_k\log(p_k).
\end{align*}
Due to this continuity, it is possible to put $\mathcal{H}_1(p)=\mathcal{H}(p)$.
We consider the R\'{e}nyi entropy as a functional  of various parameters. The first approach is to fix the distribution and consider $\mathcal{H}_\alpha(p)$ as the function of $\alpha>0.$ Some of the properties of $\mathcal{H}_\alpha(p)$ as the function of $\alpha>0$ are well known. In particular, it is   known that $\mathcal{H}_\alpha(p)$ is continuous and non-increasing in $\alpha\in(0,\infty)$, $\lim_{\alpha\to 0+} \mathcal{H}_\alpha(p)=\log m$, where $m$ is the number of non-zero probabilities, and $\lim_{\alpha\to +\infty}\mathcal{H}_\alpha(p)=-\log \max_{k} p_k$. However, for the reader's convenience, we provide the short proofs of this and some other simple statements in the Appendix. One can see that these properties of the entropy itself and its first derivative are common for all finite distributions. Alao, it is  known that R\'{e}nyi entropy is Schur concave as a function of distribution vector, that is
\begin{align*}
(p_i-p_j)\left(\frac{\partial\mathcal{H}_\alpha(p)}{\partial p_i}-\frac{\partial\mathcal{H}_\alpha(p)}{\partial p_j}\right) \leq 0, \; i\neq j.
\end{align*}
Some additional    results such as lower bounds on the difference in R\'{e}nyi entropy for distributions defined on countable alphabets could be found in \cite{Siu}. Those results usually use R\'{e}nyi divergence of order $\alpha$ of a distribution $P$ from a distribution $Q$
\begin{align*}
D_\alpha\left(P||Q\right)=\frac{1}{\alpha-1}\log\left(\sum_{i=1}^{N}\frac{p_i^\alpha}{q_i^{\alpha-1}}\right),
\end{align*}
which is very similar to Kullback-Leibler divergence. Some of R\'{e}nyi divergences most important properties were reviewed and extended in \cite{RenyiDiver}. R\'{e}nyi divergence for most commonly used univariate continuous distributions could be found in \cite{CommonDiver}. R\'{e}nyi entropy and divergence is widely used in majorization theory \cite{DiverMajor, MajorApp}, statistics \cite{BayesBound, StatMech}, information theory \cite{Siu, RenyiDiver, Compl} and many other fields. Boundedness of R\'{e}nyi entropy was shown in \cite{Bound} for discrete log-concave distributions depending on it's variance. There are other operational definitions of R\'{e}nyi entropy given in \cite{PH}, which are used in practice. However, in the present paper we restrict  ourselves with standard
R\'{e}nyi entropy and  go a step ahead in comparison with standard properties, namely, we investigate convexity of the R\'{e}nyi entropy with the help of the second derivative.  It turned out that from this point of view, the situation is much more interesting and  uncertain in comparison with the behavior of the 1st derivative, and crucially depends on the distribution. One  might say that all the standard guesses are wrong. Of course, the second derivative is continuous (evidently, it simply means that it is continuous at 1 because at all other points, the continuity is obvious), but then the surprises begin. If the second derivative starts with a positive value at zero, it can either remain positive or have inflection points, depending on the distribution. If it starts from the negative value, it can have the first infection point both before 1 and after 1, due to the distribution, too (point 1 is interesting as some crucial point for entropy, so, we compare the value of inflection points with it). The  value of the second derivative at zero is bounded from below but unbounded from above. Some superposition of entropy is convex, and this fact simultaneously describes why other similar properties depend on distribution. Due to the over-complexity of some expressions, which defied analytical consideration, we propose several illustrations performed  by numerical methods. We investigate robustness of the R\'{e}nyi entropy w.r.t. the distribution, and it turns out that the rate of respective convergence depends on initial distribution, too. Further, we establish  convergence of the disturbed entropy when the initial distribution
is uniform but the number of events increases to $\infty$
and  prove that limit of R\'{e}nyi entropy of binomial distribution is equal to entropy of Poisson distribution. It was previously proved   in \cite{BaP} that Shannon entropy of binomial distribution is increasing to entropy of Poisson distribution. Our proof of this particular fact is simpler because uses only Lebesgue's dominated convergence theorem. The paper is organized as follows. Section \ref{sec_2} is devoted to the convexity properties of the R\'{e}nyi entropy, Section \ref{sec_3} describes  robustness of the   R\'{e}nyi entropy, and Section \ref{sec_4} contains some auxiliary results.

\section{Convexity of the R\'{e}nyi entropy}\label{sec_2}
To start, we consider the general properties of the 2nd derivative of the R\'{e}nyi entropy.
\subsection{The form and the continuity  of the 2nd derivative}
Let's denote $S_i(\alpha)=\sum_{k=1}^{N}p_k^\alpha\log^ip_k, i=0,1,2,3. $ Denote also $f(\alpha)=\log\left( \sum_{k=1}^{N} p_k^{\alpha}\right)$. Obviously, function $f\in C^\infty(\mathbb{R^+})$, and its first derivatives equal
 $$f'(\alpha)=\frac{S_1(\alpha)}{S_0(\alpha)},\; f''(\alpha)=\frac{S_2(\alpha)S_0(\alpha)-S^2_1(\alpha)}{S^2_0(\alpha)},$$ $$f'''(\alpha)
 =\frac{S_3(\alpha)S^2_0(\alpha)-3S_2(\alpha)S_1(\alpha)S_0(\alpha)+2S^3_1(\alpha)}{S^3_0(\alpha)}.$$ In particular, if to consider the random variable $\xi$ taking values $\log p_k$ with probability $p_k$, then
 \begin{equation}\begin{gathered} \label{equ_3}
    f'(1)=E(\xi)<0,\;f''(1)=E(\xi^2)-(E(\xi))^2>0,\\
 f'''(1)=E(\xi^3)-3E(\xi^2)E(\xi)+2(E(\xi))^3,
 \end{gathered}\end{equation} and the sign of $f'''(1)$ is not clear (as  we can see below, it can be both $+$ and $-$).
\begin{lemma}
 Let $p_k\neq0$ for all $1\leq k\leq N$. Then
\begin{itemize}
\item[(i)] (a) The 2nd derivative $\mathcal{H}^{''}_{\alpha}(p)$ equals
\begin{align}\label{equ_1}
\mathcal{H}^{''}_{\alpha}(p)=-\frac{1}{(1-\alpha)^3} \left(\sum_{k=1}^{N} \left((1-\alpha) q_k'(\alpha)+2q_k(\alpha)\right)\log\frac{q_k(\alpha)}{p_k} \right),
\end{align}
where $$q_k(\alpha)=\frac{p_k^\alpha}{\sum_{k=1}^{N} p_k^\alpha}.$$

(b) The 2nd derivative $\mathcal{H}^{''}_{\alpha}(p)$ can be also presented as
  \begin{align}\label{equ_2}\mathcal {H}^{''}_{\alpha}(p)=-\frac{1}{3}f'''(\theta)\end{align}
  for some $0<\theta<\alpha.$
\item[(ii)] The 2nd derivative  $\mathcal{H}^{''}_{\alpha}(p)$ is continuous on $\mathbb{R}^+$ if we put $$\mathcal{H}^{''}_{\alpha}(1)=-\frac{1}{3}f'''(1)=-\frac{1}{3}(E(\xi^3)-3E(\xi^2)E(\xi)+2(E(\xi))^3).$$
\end{itemize}
\end{lemma}
\begin{proof} Equaity \eqref{equ_1} is a result of direct calculations. Concerning equality \eqref{equ_2}, we can present
$\mathcal {H}_\alpha(p)$ as
$$\mathcal {H}_\alpha(p)=\frac{f(\alpha)-f(1)}{1-\alpha},$$
therefore, $-\mathcal {H}_\alpha(p)$ is a slope function for $f$. Taking successive derivatives, we get from standard Taylor formula that
$$\mathcal{H}^{'}_{\alpha}(p)=\frac{f'(\alpha)(1-\alpha)+f(\alpha)}{(1-\alpha)^2}=-\frac{1}{2}f^{''}(\eta),$$
and  $$\mathcal{H}^{''}_{\alpha}(p)=\frac{f''(\alpha)(1-\alpha)^2+2f'(\alpha)}{(1-\alpha)+2f(\alpha)}=-\frac{1}{3}f^{'''}(\theta),$$
where $\eta,\theta\in(0,\alpha)$. If $\alpha\rightarrow 1$, then both $\eta$ and $\zeta$ tend to 1. Taking into account \eqref{equ_3},  we immediately get both equality \eqref{equ_2} and statement $(ii)$.
\end{proof}

\subsection{Behavior of the 2nd derivative at the origin} Let us consider the starting point for the 2nd derivative, i.e., the behavior of $\mathcal{H}^{''}_{\alpha}(p)$ at zero as a function of a distribution vector $p$. Analyzing \eqref{equ_1}, we see that
$\mathcal{H}^{''}_{\alpha}(p)$ as function of $\alpha$ is continuous in 0. Moreover,
$$q_k(0)=1/N,\;q_k'(0)=\frac{\log p_k}{N}-\frac{\sum_{k=1}^N\log p_k}{N^2},$$ so we can present $\mathcal{H}^{''}_{\alpha}(p)$ as
\begin{align*}
\mathcal{H}^{''}_{0}(p)&=-\sum_{k=1}^{N} \left(\frac1N\log p_k-\frac{1}{N^2}\sum_{i=1}^{N} \log p_i+\frac2N \right)\log\frac{1}{Np_k}
\\ &=
\sum_{k=1}^{N} \left(\frac1N\log p_k-\frac{1}{N^2}\sum_{i=1}^{N} \log p_i+\frac2N \right)\left(\log N+\log p_k\right)
\\ &=
2\log N +\frac1N\sum_{k=1}^{N}(\log p_k)^2-\frac{1}{N^2}\left(\sum_{k=1}^{N} \log p_k\right)^2 +\frac2N\sum_{k=1}^{N} \log p_k.
\end{align*}
Now we are interested in the sign of $\mathcal{H}^{''}_{\alpha}(p)$. Give an example of distributions for which $\mathcal{H}^{''}_{\alpha}(p)>0 $ it is very simple, one of such examples is given at Figure 1. Concerning negative $\mathcal{H}^{''}_{\alpha}(p)$, it is also possible, however, at this moment  we prefer to start with a more general result.
\begin{lemma}
If some probability vector is $p$ a point of local extremum of $\mathcal{H}^{''}_{\alpha}(p)$ then either $p=p(uniform)=\left(\frac{1}{N},\ldots,\frac{1}{N}\right)$ or it contains   two different probabilities.
\end{lemma}
\begin{proof}
Let us formulate the  necessary conditions for   $\mathcal{H}^{''}_{0}(p)$ to have a local   extremum at some point. Taking into account limitation  $\sum_{k=1}^{N} p_k=1$, these conditions have a form
\begin{align*}
\begin{cases}
2\log N +\frac1N\sum_{k=1}^{N}(\log p_k)^2-\frac{1}{N^2}\left(\sum_{k=1}^{N} \log p_k\right)^2 +\frac2N\sum_{k=1}^{N} \log p_k \longrightarrow extr\\
\sum_{k=1}^{N} p_k=1.
\end{cases}
\end{align*}
We create a  Lagrangian function \begin{equation*}\begin{gathered}
L = \lambda_0\left(2\log N +\frac1N\sum_{k=1}^{N}(\log p_k)^2-\frac{1}{N^2}\left(\sum_{k=1}^{N} \log p_k\right)^2 +\frac2N\sum_{k=1}^{N} \log p_k\right)\\+\lambda\left(\sum_{k=1}^{N} p_k-1\right). \end{gathered} \end{equation*}
If some $p$ is an extreme point then there exist  $\lambda_0$ and $\lambda$ such that $\lambda_0^2+\lambda^2\neq0$ and $\frac{\partial L}{\partial p_i}(p)=0$ for all $1\leq i\leq N$, i.e.,
\begin{align*}
\frac{\partial L}{\partial p_i} = \lambda_0\left(\frac{2}{Np_i}\log p_i -\frac {2}{N^2p_i}\left(\sum_{k=1}^{N} \log p_k\right)+\frac{2}{Np_i}\right)+\lambda=0.
\end{align*}
If $\lambda_0=0$ then $\lambda=0$. However,  $\lambda_0^2+\lambda^2\neq0$, therefore we can put $\lambda_0=1$. Then

$$-\lambda p_i = \frac{2}{N}\log p_i -\frac {2}{N^2}\left(\sum_{k=1}^{N} \log p_k\right)+\frac{2}{N}.$$
taking a sum of these equalities we get that $\lambda=-2$ whence
\begin{equation}\label{equ_twoprob}p_i-\frac1N\log p_i=\frac1N-\frac{1}{N^2}\left(\sum_{k=1}^{N} \log p_k\right).\end{equation}

So, if distribution vector $p$ is an extreme point  then $p_1-\frac1N\log p_1=...=p_N-\frac1N\log p_N$. Let's have a look at continuous function $f(x)=x-\frac1N\log x, \; x\in(0,1)$. Its derivative equals
\begin{align*}
f'(x)=1-\frac{1}{Nx}=0 \Leftrightarrow x=\frac1N, \;
\sign(f'(x))=\sign\left(x-\frac1N\right),
\end{align*}
\begin{align*}
\lim_{x\to0+}f(x)=+\infty,\;\lim_{x\to+1}f(x)=1.
\end{align*}
So, $f(x)$ has its global minimum at point $x=\frac1N$, and for any $ f(\frac1N)<y\le 1$ there exist two points,  $x'\neq x'',\; x', x''\in(0,1)$ such that $f(x')=f(x'')=y$. Thus, if the vector of probabilities is a vector of local extremum of $\mathcal{H}^{''}_{0}(p)$, then it contains no more than two  different probabilities. Obviously, it can be $p=p(uniform)=\left(\frac{1}{N},\ldots,\frac{1}{N}\right)$.\end{proof}
\begin{remark}
Note that  $\mathcal{H}^{''}_{0}(p(uniform))=0.$ Therefore, in order to find the distribution for which $\mathcal{H}^{''}_{0}(p)<0 $ let us consider the distribution vector that contains only two different probabilities $p_0, q_0$ such that:
\begin{equation}\label{distr_cond}
\begin{cases}
p_0-q_0 = \frac1N\left(\log p_0-\log q_0\right), \\
kp_0+(N-k)q_0=1,
\end{cases}
\end{equation}
where $N, k \in \mathbb{N}, \; N>k$ and $p_0, q_0 \in (0,1)$.
\end{remark}

\begin{lemma}\label{negin0}
Let $p$ be distribution vector satisfying \eqref{distr_cond}. Then $\mathcal{H}^{''}_{0}(p)<0$.
\end{lemma}
\begin{proof}
First, we will show that $\mathcal{H}^{''}_{0}(p)$ is non-positive. For that we rewrite $\mathcal{H}^{''}_{0}(p)$ in terms of $p_0$ and $q_0$:
\begin{align*}
\mathcal{H}^{''}_{0}(p)&= 2\log N+\frac1N\left(k(\log p_0)^2+(N-k)(\log q_0)^2\right)-
\frac{1}{N^2}\left(k\log p_0+(N-k)\log q_0\right)^2
\\ &+
\frac2N\left(k\log p_0 + (N-k)\log q_0\right) \\& = 2\log N+\frac{k(N-k)}{N^2}
\left((\log p_0)^2-2\log p_0\log q_0+(\log q_0)^2\right)
\\ &+
\frac{2k}{N}\left(\log p_0-\log q_0\right) + 2\log q_0 = 2\log Nq_0 +k(N-k)(p_0-q_0)^2+2k(p_0-q_0).
\end{align*}
We know that $kp_0+(N-k)q_0=1$, whence $k = \frac{ Nq_0-1}{q_0-p_0}$, and
$N-k=\frac{1-Np_0 }{q_0-p_0}$. Then
\begin{align*}
\mathcal{H}^{''}_{0}(p) &= 2\log Nq_0+(1-Nq_0)(Np_0-1)+2(1-Nq_0)\\ & = 2 \log Nq_0 + N(p_0-q_0)+1-N^2p_0q_0
\\ &=
\log (Nq_0)^2+\log\frac{p_0}{q_0}+1-N^2p_0q_0=\log N^2p_0q_0-N^2p_0q_0+1.
\end{align*}
Note that $\log x -x+1<0$ for $x>0, x\neq 1$. We want to show that under conditions \eqref{distr_cond} $N^2p_0q_0$ can't be equal to 1. Suppose that $N^2p_0q_0=1$. Then it follows from \eqref{distr_cond} that
\begin{align*}
\frac{k}{N^2}+(N-k)q_0^2=q_0.
\end{align*}
It means that $q_0$ and $p_0$ are algebraic numbers. Thus, their defference $p_0-q_0$ is also algebraic. On the other hand, by Lindemann–Weierstrass theorem $\frac1N\left(\log p_0-\log q_0\right)$ is transcendental number, which contradicts \eqref{distr_cond}. So $N^2p_0q_0\neq1$ and $\mathcal{H}^{''}_{0}(p)<0$.
\end{proof}

\begin{theorem}
For any $n>2$ there exists $N\geq n$ and a probability vector $p=(p_1,\ldots,p_N)$ such that $\mathcal{H}^{''}_{0}(p)<0.$
\end{theorem}
\begin{proof}
Consider the distribution vector $p$ that satisfies conditions \eqref{distr_cond}. From Lemma \ref{negin0} we know that $\mathcal{H}^{''}_{0}(p)<0$. Now we want to show that there exist arbitrarily large $N\in \mathbb{N}$ and distribution vector p of length N that satisfy those conditions. For that we denote
\begin{align*}
x=Np_0, \; y=Nq_0, \; r=\frac kN=\frac{ y-1}{ y-x}.
\end{align*}
Then $0<x<1<y$ and $r<1$ and $x-y=\log x-\log y$. Function $x - \log x$ is decreasing on $(0,1)$, is increasing on $(1,+\infty)$ and is equal to 1 at point 1. Let $y=y(x)$ be implicit function defined by $x-y=\log x-\log y$. By that we get 1-to-1 correspondence from $x\in(0,1)$ to $y\in(1,+\infty)$. We also have fuction $r(x)=\frac{y(x)-1}{y(x)-x}$. If we find $x'\in(0,1)$ such that $r'=r(x')$ is rational then we can pick $N, k \in \mathbb{N}$ such that $r'=\frac{k}{N}$ and get distribution vector p satisfying \eqref{distr_cond} with $p_0=\frac{x}{N}, \; q_0=\frac{y}{N}$. However, we won't find such $x'$, we will just show that they exist. To do that observe that $y(x)$ is continuous function of $x$ and so is function $r(x)=\frac{y(x)-1}{y(x)-x}$. What's more,
$$y(x)\to+\infty, \; x\to0+ \; so \; r(x)\to1, \; x\to0+.$$
Let's fix $x_0 \in (0,1), \; r(x_0)<1$. Then for any $r'\in(r(x_0),1)$ there exists $x'\in(0,x_0)$ such that $r(x')=r'$. By taking $r'\in\mathbb{Q}$ we get that there exists $x'$ such that $\frac{k}{N}<1$ and is rational. Finally, we want to show that $N$ can be arbitrarily large. For that simply observe that $\frac{k}{N}=r'$ so as $r'\to1-$ we get that $N\to +\infty$.
\end{proof}

\begin{lemma}
Let N be fixed. Then $\mathcal{H}^{''}_{0}(p)$ as the function of vector $p$ is bounded from below and   is unbounded from above.
\end{lemma}
\begin{proof} Recall that $\mathcal{H}^{''}_{0}(p)=0$ on the uniform distribution and exclude this case from further  consideration.
In order to simplify the notations,  we denote $x_k=\log p_k$,   and let
\begin{align*}
S_N:= N(\mathcal{H}^{''}_{0}(p)-2\log N)=\sum_{k=1}^{N} (x_k)^2-\frac1N\left(\sum_{k=1}^{N} x_k \right)^2+2\sum_{k=1}^{N} x_k.
\end{align*}
  Note that there exists $n\leq N-1$ such that $$x_1<\log\frac1N, ..., x_n<\log\frac1N, \; x_{n+1}\geq\log\frac1N, ..., x_N\geq\log\frac1N.$$ Further, denote the rectangle $A=[\log\frac1N;0]^{N-n}\subset \mathbb{R}^{N-n}$, and let $$S_{N,1} = \sum_{k=1}^{n} x_k, \; S_{N,2} = \sum_{k=n+1}^{N} x_k.$$ Let's establish that   $\mathcal{H}^{''}_{0}(p)$ is bounded from below.
  In this connection, rewrite $S_N$ as
\begin{align*}
S_N = \sum_{k=1}^{n} x_k^2+\sum_{k=n+1}^{N} x_k^2-\frac1N\left((S_{N,1})^2+2S_{N,1}S_{N,2}+(S_{N,2})^2 \right)+2S_{N,1}+2S_{N,2}.
\end{align*}
By Cauchy–Schwarz inequality we have $$\left(\sum_{k=1}^{n}x_k\right)^2 \leq n\sum_{k=1}^{n} x_k^2, \; \left(\sum_{k=n+1}^{N}x_k\right)^2 \leq (N-n)\sum_{k=n+1}^{N} x_k^2.$$
Therefore
\begin{align*}
S_N &\geq \left(1-\frac{n}{N}\right)\sum_{k=1}^{n} x_k^2+\frac{n}{N}\sum_{k=n+1}^{N} x_k^2-\frac2NS_{N,1}S_{N,2}+2S_{N,1}+2S_{N,2}
\\ &=
\sum_{k=1}^{n} \left(\left(1-\frac{n}{N}\right)x_k^2+x_k\left(2-\frac2NS_{N,2}\right)\right)+\frac{n}{N}\sum_{k=n+1}^{N} x_k^2+2S_{N,2}
\\ &=
\frac1N \sum_{k=1}^{n} \left(\left(N-n\right)x_k^2+2x_k\left(N-S_{N,2}\right)\right)+\frac{n}{N}\sum_{k=n+1}^{N} x_k^2+2 S_{N,2}.
\end{align*}
There exists $M>0$ such that for every $n\leq N-1$ we have $|S_{N,2}|\leq M$ because $A$ is compact and $S_{N,2}$ is continuous on $A$. Obviously, $\frac{n}{N}\sum_{k=n+1}^{N} x_k^2\ge 0$. Finally, for every $1\leq k\leq n$ we have that $\left(N-n\right)x_k^2+2x_k\left(N-S_2\right)$ is bounded from below by the value $-\frac{(N-S_{2,N})^2}{N-n}\ge -N^2-M^2.$  Resuming, we get that $S_N$ is bounded from below, and consequently  $  \mathcal{H}^{''}_{0}(p)$ is bounded from below for fixed $N$.

Now we want to establish  that $\mathcal{H}^{''}_{0}(p)$ is not bounded from above. In this connection, let $\varepsilon>0$, and  let us consider the distribution of the form $p_1=\varepsilon, \; p_2=...=p_N=\frac{1-\varepsilon}{N-1}$. Then we have
\begin{align*}
\mathcal{H}^{''}_{0}(p) &=2\log N +\frac1N\sum_{k=1}^{N}(\log p_k)^2-\frac{1}{N^2} \left(\sum_{k=1}^{N} \log p_k\right)^2 +\frac2N\sum_{k=1}^{N} \log p_k
\\ &=
2\log N+\frac{N-1}{N}\left(\log\frac{1-\varepsilon}{N-1}\right)^2+ \frac1N(\log\varepsilon)^2 -\frac{1}{N^2}\left((N-1)\log\frac{1-\varepsilon}{N-1} +\log\varepsilon\right)^2
\\ &+
\frac{2(N-1)}{N}\log\frac{1-\varepsilon}{N-1}+\frac2N\log\varepsilon = \left(\frac1N-\frac{1}{N^2}\right)\left(\log\varepsilon\right)^2
\\ &+
\left(\frac2N -\frac{2(N-1)}{N^2}\log\frac{1-\varepsilon}{N-1}\right) \log\varepsilon+2\log N + \left(\frac{N-1}{N}-\frac{(N-1)^2}{N^2}\right) \left(\log\frac{1-\varepsilon}{N-1}\right)^2
\\ &+
\frac{2(N-1)}{N}\log\frac{1-\varepsilon}{N-1}\to+\infty,\;\varepsilon\to0+.
\end{align*}
\end{proof}
\subsection{Superposition of entropy that is  convex}
Now  we establish that  the superposition of entropy with some decreasing function is   convex. Namely, we shall consider function
\begin{equation}\label{eq_4}
 \mathcal {G}_{\beta}(p)=-\mathcal {H}_{1+\frac{1}{\beta}}(p)=\beta\log\left(\sum_{k=1}^N p_k^{1+1/\beta}\right),\,\beta>0 \end{equation}
and prove its convexity. Since now we consider the tools that do not include differebtiation, we can assume that some probabilities are zero. In order to provide convexity, we start with the following simple and known result whose proof is added for the reader's convenience.

\begin{lemma}\label{Lem_2}
For any measure space $(\mathcal  X, \Sigma, \mu)$ and any measurable $f\in L^p(\mathcal  X, \Sigma, \mu)$ for some interval $p\in[a,b]$, $\norm{f}_p = \norm{f}_{L^p(\mathcal  X, \Sigma, \mu)}$ is log-convex as a function of $1/p$ on this interval.
\end{lemma}

\begin{proof}
For any $p_2,p_1>0$ and $\theta\in (0,1)$, denote $p = \big( \theta /p_1 + (1-\theta)/p_2\big)^{-1}$ and observe that
$$\theta p/p_1 + (1-\theta) p/p_2 = 1.$$
Therefore, by the H\"older inequality
\begin{align*}
\norm{f}^p_{p} &= \int_{\mathcal  X} |f(x)|^{\theta p}\cdot |f(x)|^{(1-\theta) p} \mu(dx) \\
&\le \left(\int_{\mathcal  X} |f(x)|^{p_1} \mu(dx) \right)^{\theta p / p_1} \left(\int_{\mathcal  X} |f(x)|^{p_2} \mu(dx) \right)^{(1-\theta) p / p_2},
\end{align*}
whence
$$\log \norm{f}_p \le \theta \log \norm{f}_{p_1} + (1-\theta)\log \norm{f}_{p_2}, $$
as required.
\end{proof}
\begin{corollary}For  any probability vector $p=(p_k,1\leq k\leq N)$,   function
 $ \mathcal{G}_\beta(p), \beta>0,$
is convex.
\end{corollary}
\begin{proof}
Follows from Lemma \ref{Lem_2} by setting $\mathcal  X = \set{1,\dots,N}$, $\mu(A) = \sum_{k\in A} p_k$, $f(k) = p_k$, $k\in \mathcal  X$.
\end{proof}

\begin{remark}
  It follows immediately from \eqref{eq_4} that   for the function
$$\mathcal {G}_\beta(p) = \beta \log \sum_{k=1}^N p_k^{1+1/\beta},\;\beta>0$$
then $\mathcal {H}_\alpha(p)=\mathcal {G}_{\frac{1}{\alpha-1}}(p)$. For $\alpha>1 \; \frac{1}{\alpha-1}$ is convex. If it happened that there is such $p$ that $G_{\cdot}(p)$ is non-decreasing on an interval then $\mathcal {G}_{\frac{1}{\alpha-1}}(p)$ be  convex on that interval and $ \mathcal {H}_\alpha(p)$ be convex, too. However,
\begin{align*}\begin{gathered}
\mathcal {G}'_{\beta}(p)  = \log \sum_{k=1}^N p_k^{1+1/\beta}-\frac{1}{\beta}\frac{\sum_{k=1}^N p_k^{1+1/\beta}\log p_k}{\sum_{k=1}^N p_k^{1+1/\beta}}\\
= -\sum_{k=1}^N \frac{p_k^{1+1/\beta}}{\sum_{k=1}^N p_k^{1+1/\beta}} \log \frac{p_k^{ 1/\beta}}{ \sum_{k=1}^N p_k^{1+1/\beta}}\leq 0.
\end{gathered}\end{align*}
In some sense, this is a reason why we can not say something definite concerning the 2nd derivative of entropy either on the whole semiaxes or even in the interval $[1,+\infty).$
\end{remark}

\subsection{Graphs of $\mathcal{H}_\alpha(p)$ and it's second derivative of several probability distributions}

\begin{figure}[b]
\begin{center}
\begin{subfigure}[h]{0.45\linewidth}
\includegraphics[width=1\linewidth]{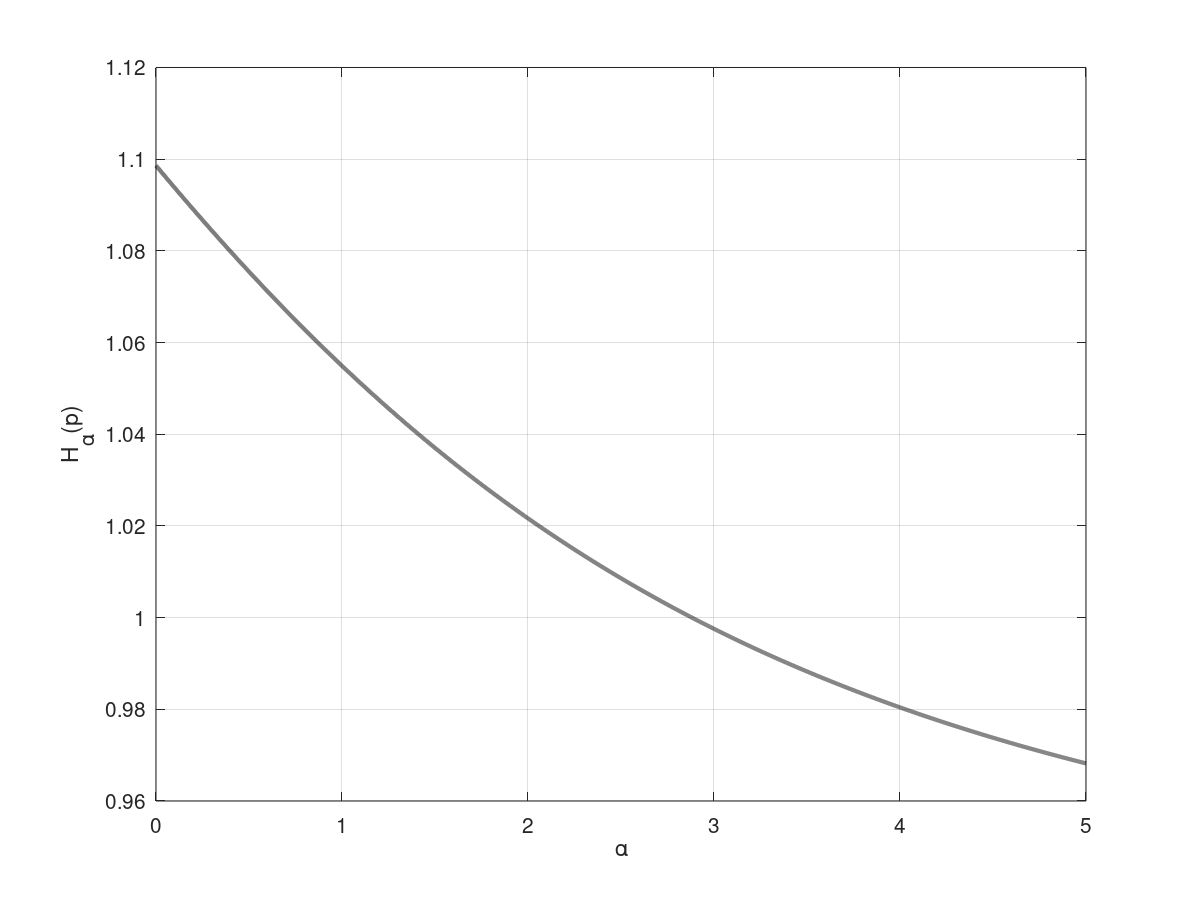}\label{fig:conv}
\end{subfigure}%
\begin{subfigure}[h]{0.45\linewidth}
\includegraphics[width=1\linewidth]{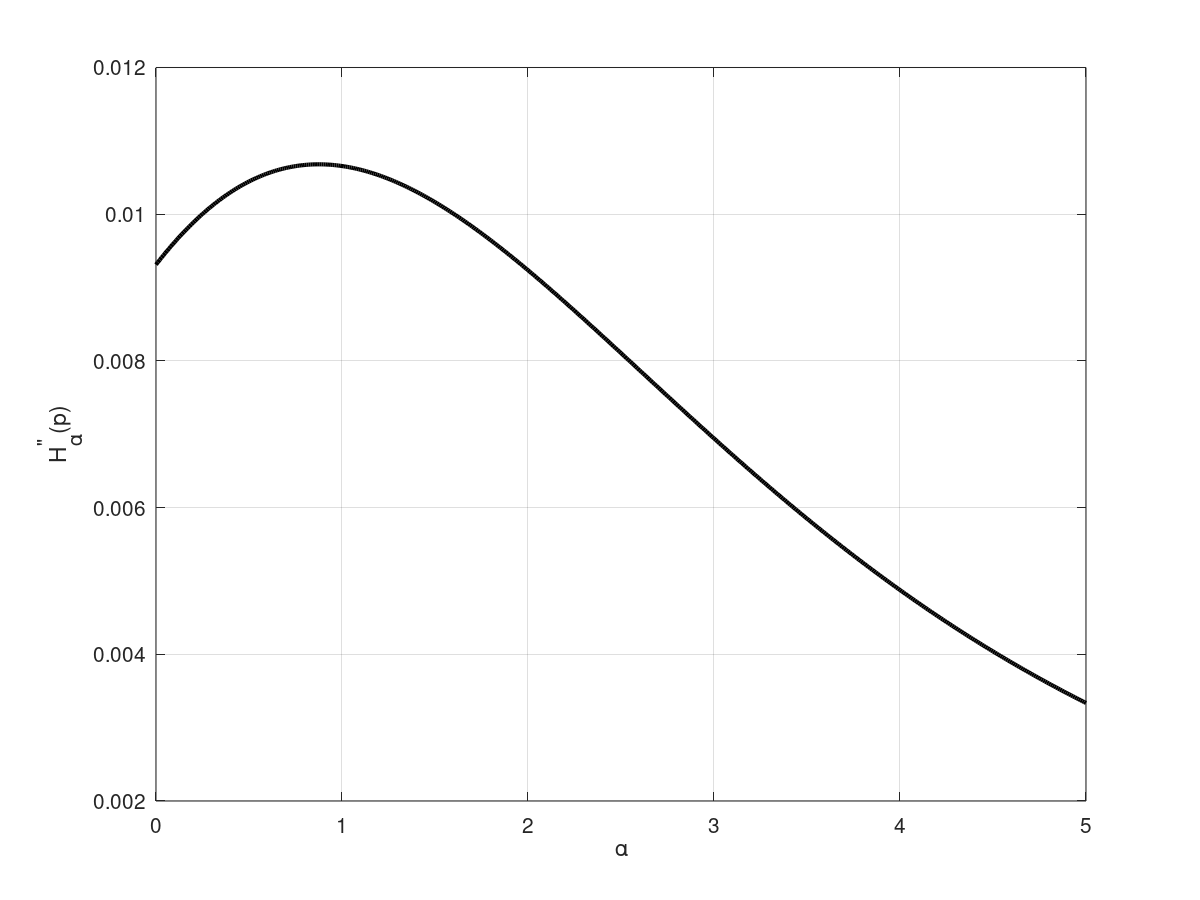}\label{fig:conv_der}
\end{subfigure}
\caption{Graph of $\mathcal{H}_\alpha(p)$ and  $\mathcal{H}^{''}_{\alpha}(p)$, where $p_1=p_{2}=0.4, p_{3}=0.2$. Here $\mathcal{H}_\alpha(p)$ is convex as a function of $\alpha>0$.}
\end{center}
\end{figure}

\begin{figure}[h]
\begin{center}
\begin{subfigure}[h]{0.45\linewidth}
\includegraphics[width=1\linewidth]{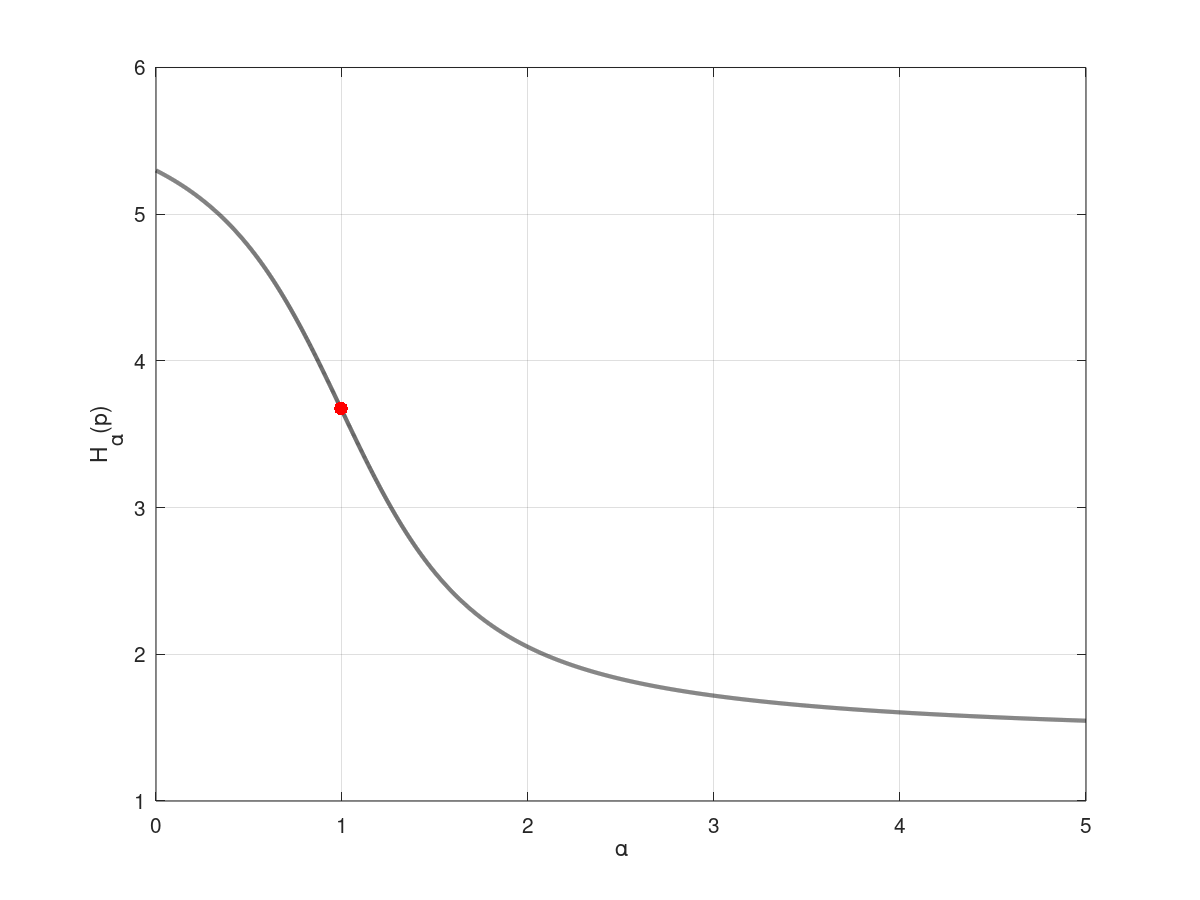}\label{fig:conc}
\end{subfigure}%
\begin{subfigure}[h]{0.45\linewidth}
\includegraphics[width=1\linewidth]{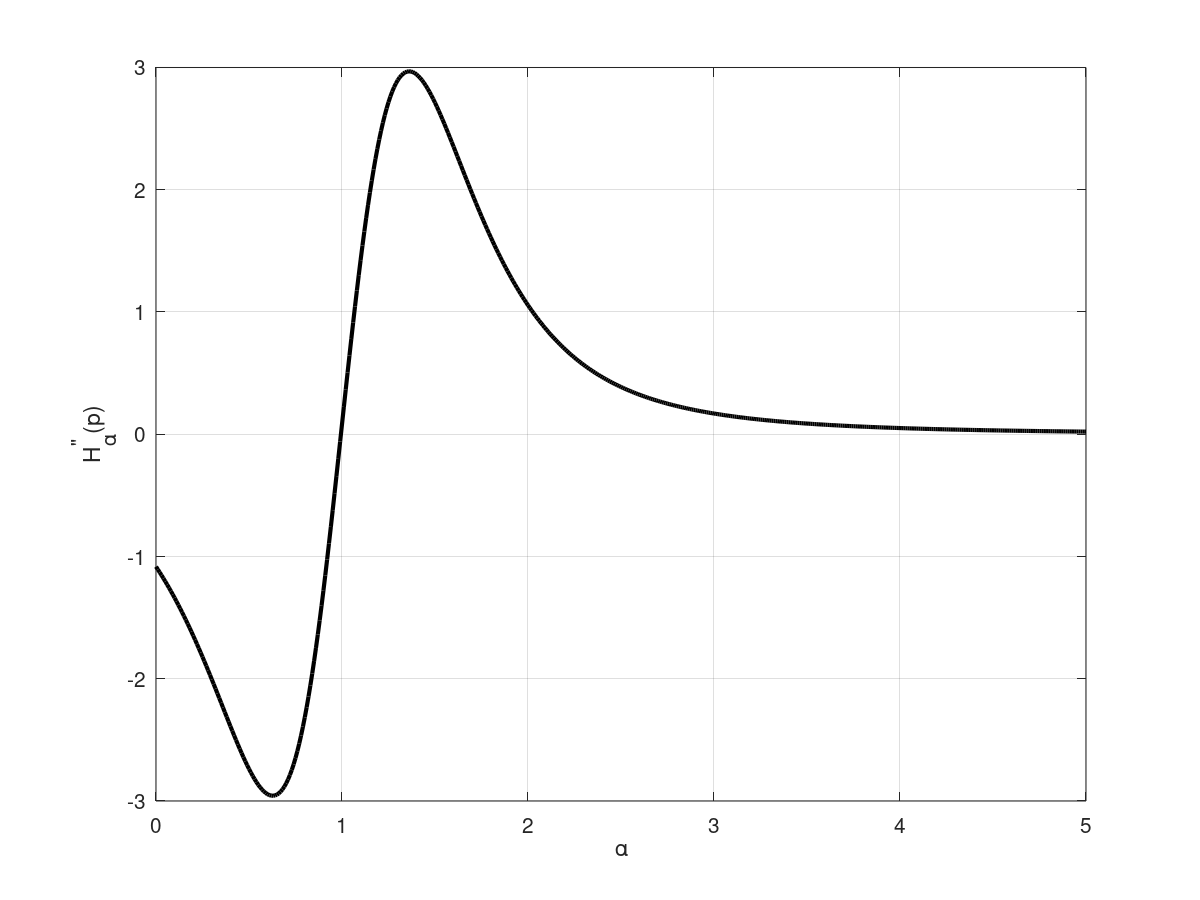}\label{fig:conc_der}
\end{subfigure}
\caption{Graph of $\mathcal{H}_\alpha(p)$ and  $\mathcal{H}^{''}_{\alpha}(p)$, where $p_1=...=p_{198}=\frac{1}{400}, p_{199}=p_{200}=\frac{101}{400}$. Dot is the point where $H_\alpha^{''}(p)=0$  and this point is $\alpha=0.99422$.}
\end{center}
\end{figure}

\begin{figure}[h]
\begin{center}
\begin{subfigure}[h]{0.45\linewidth}
\includegraphics[width=1\linewidth]{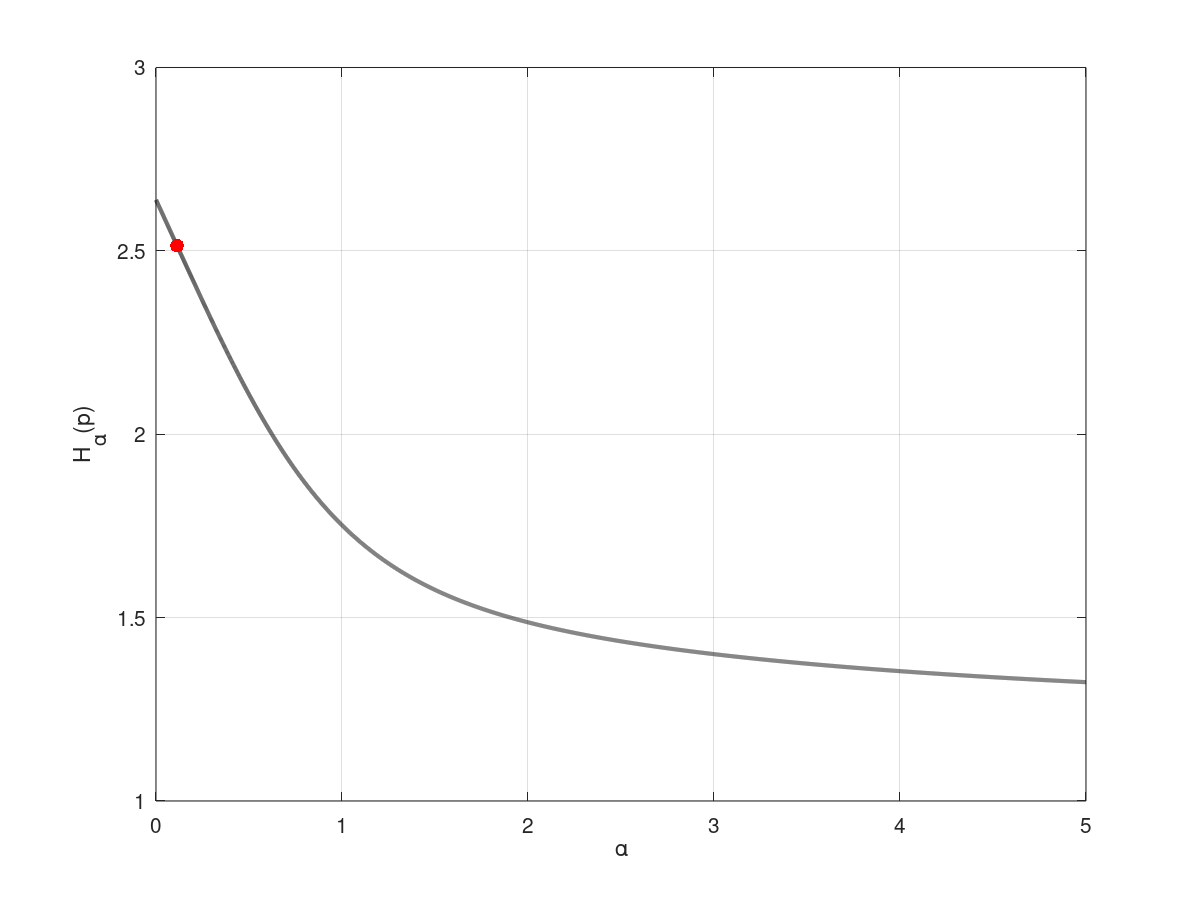}\label{fig:b1}
\end{subfigure}%
\begin{subfigure}[h]{0.45\linewidth}
\includegraphics[width=1\linewidth]{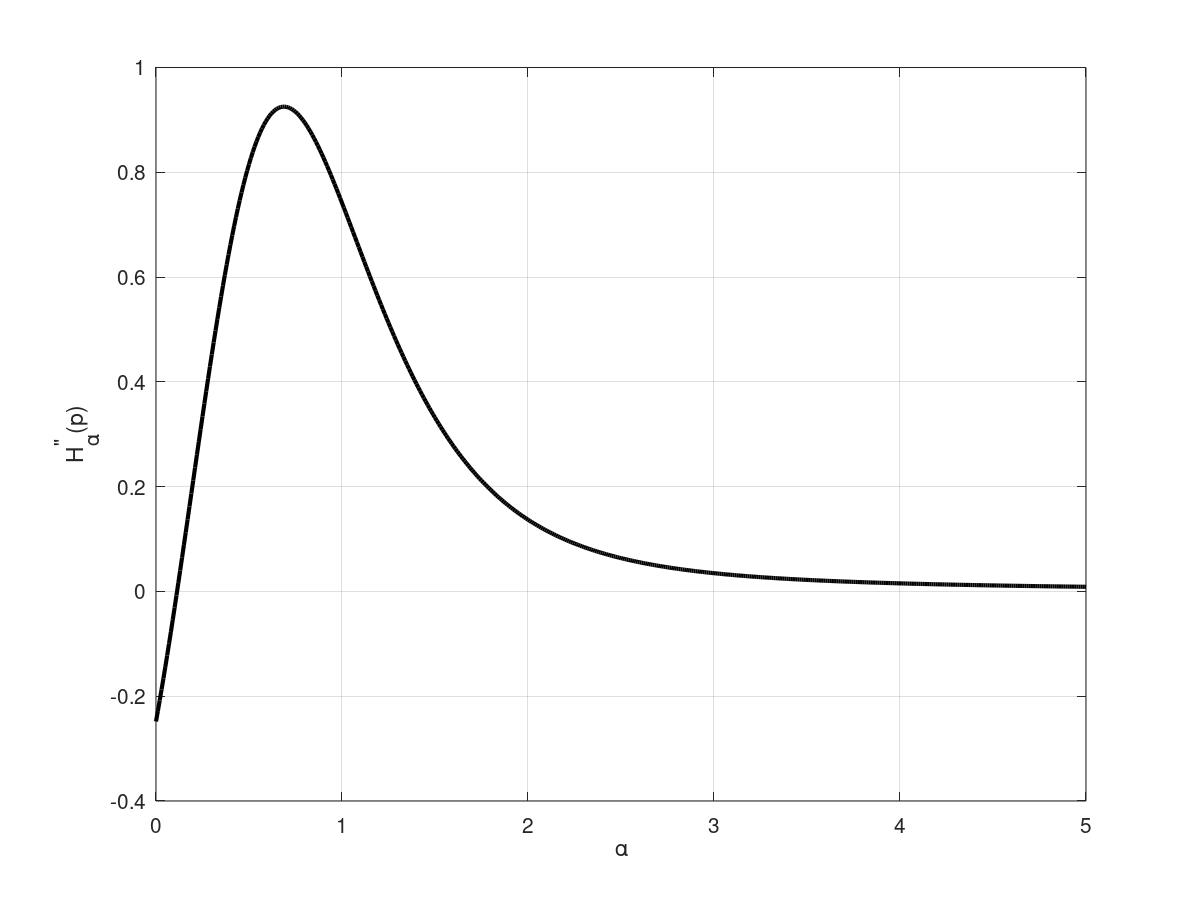}\label{fig:b1_der}
\end{subfigure}
\caption{Graph of $\mathcal{H}_\alpha(p)$ and $\mathcal{H}^{''}_{\alpha}(p)$, where $p_1=...=p_{10}=0.01, p_{11}=p_{12}=0.15, p_{13}=p_{14}=0.3$. Here second derivative becomes positive long before point 1 (at point 0.11318).}
\end{center}
\end{figure}

\begin{figure}[h]
\begin{center}
\begin{subfigure}[h]{0.45\linewidth}
\includegraphics[width=1\linewidth]{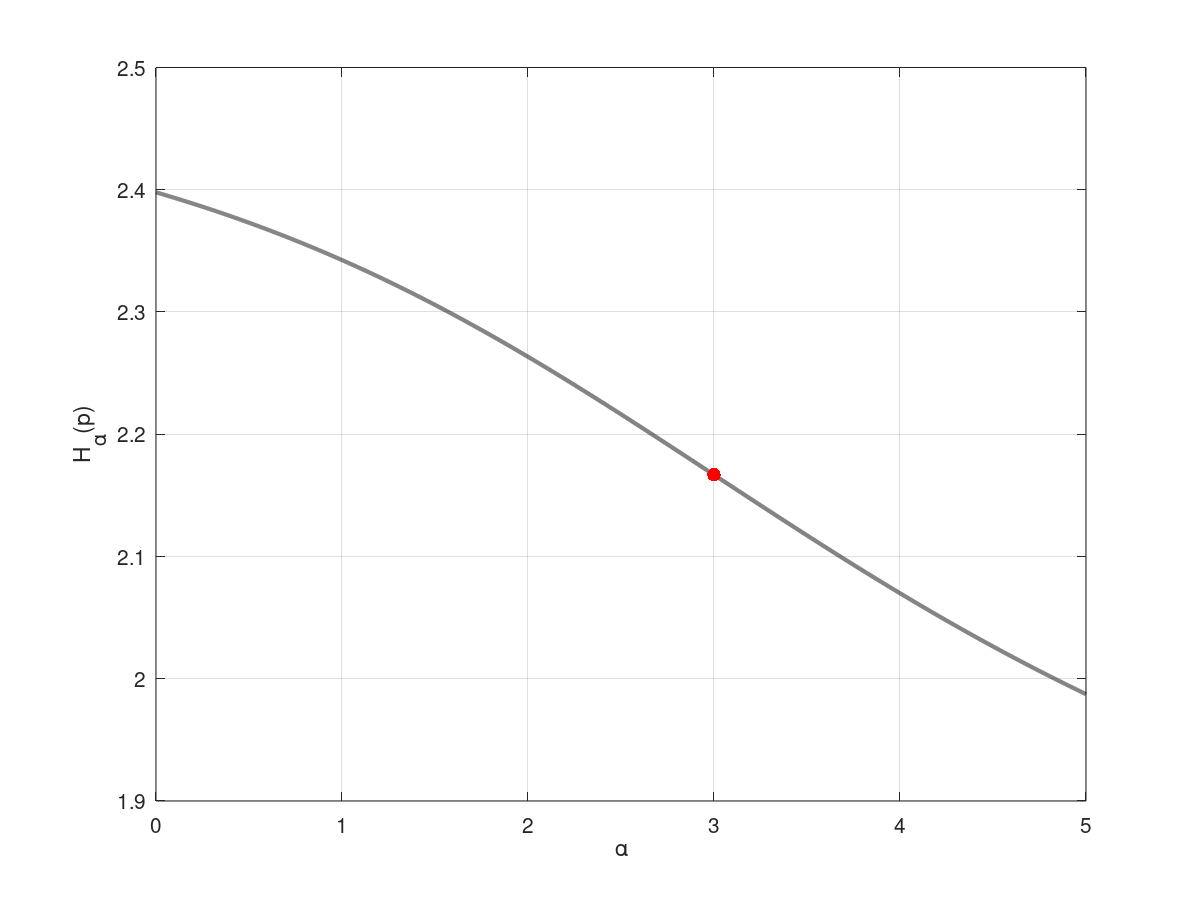}\label{fig:a1}
\end{subfigure}%
\begin{subfigure}[h]{0.45\linewidth}
\includegraphics[width=1\linewidth]{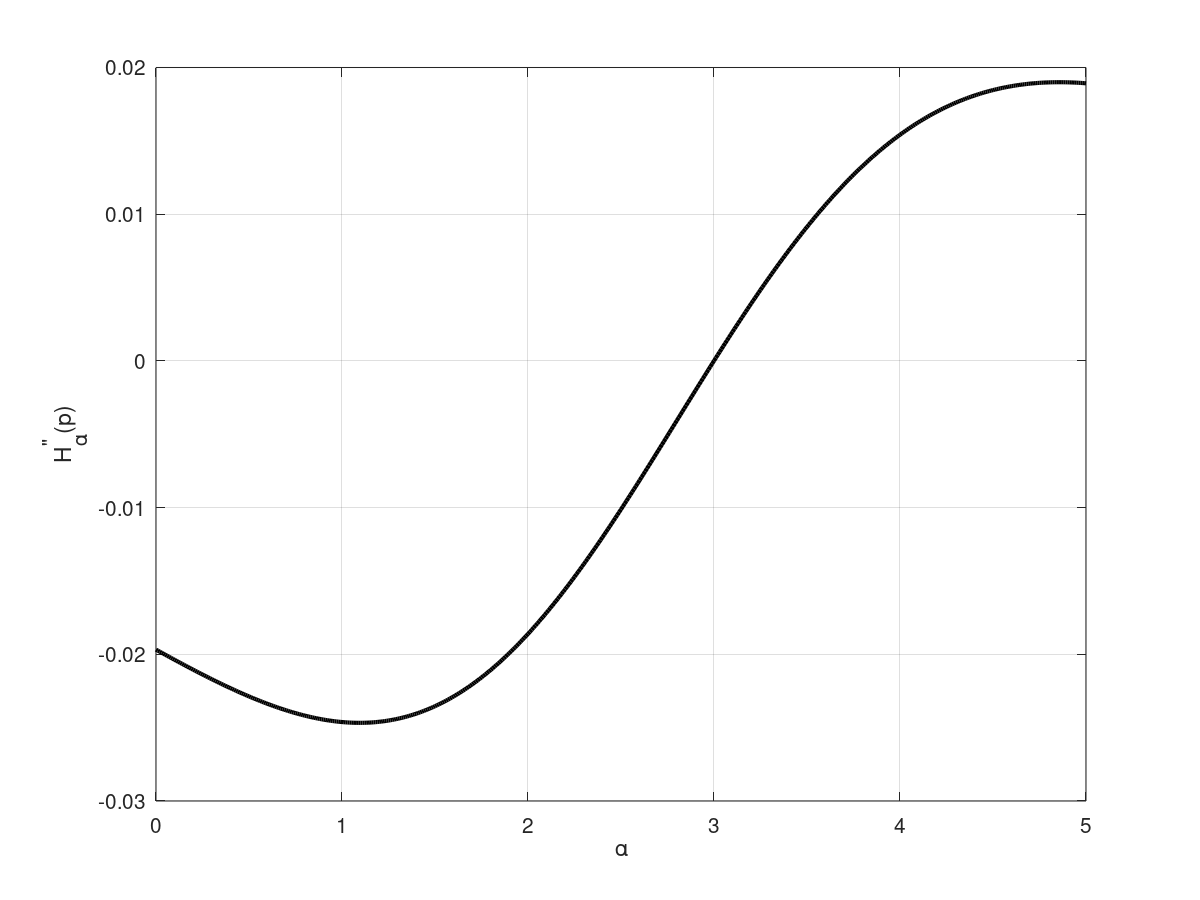}\label{fig:a1_der}
\end{subfigure}
\caption{Graph of $\mathcal{H}_\alpha(p)$ and $\mathcal{H}^{''}_{\alpha}(p)$, where $p_1=...=p_{10}=0.08, p_{11}=0.2$. Here second derivative becomes positive after point 1 (at point 2.9997).}
\end{center}
\end{figure}

\begin{figure}[h]
\begin{center}
\begin{subfigure}[h]{0.45\linewidth}
\includegraphics[width=1\linewidth]{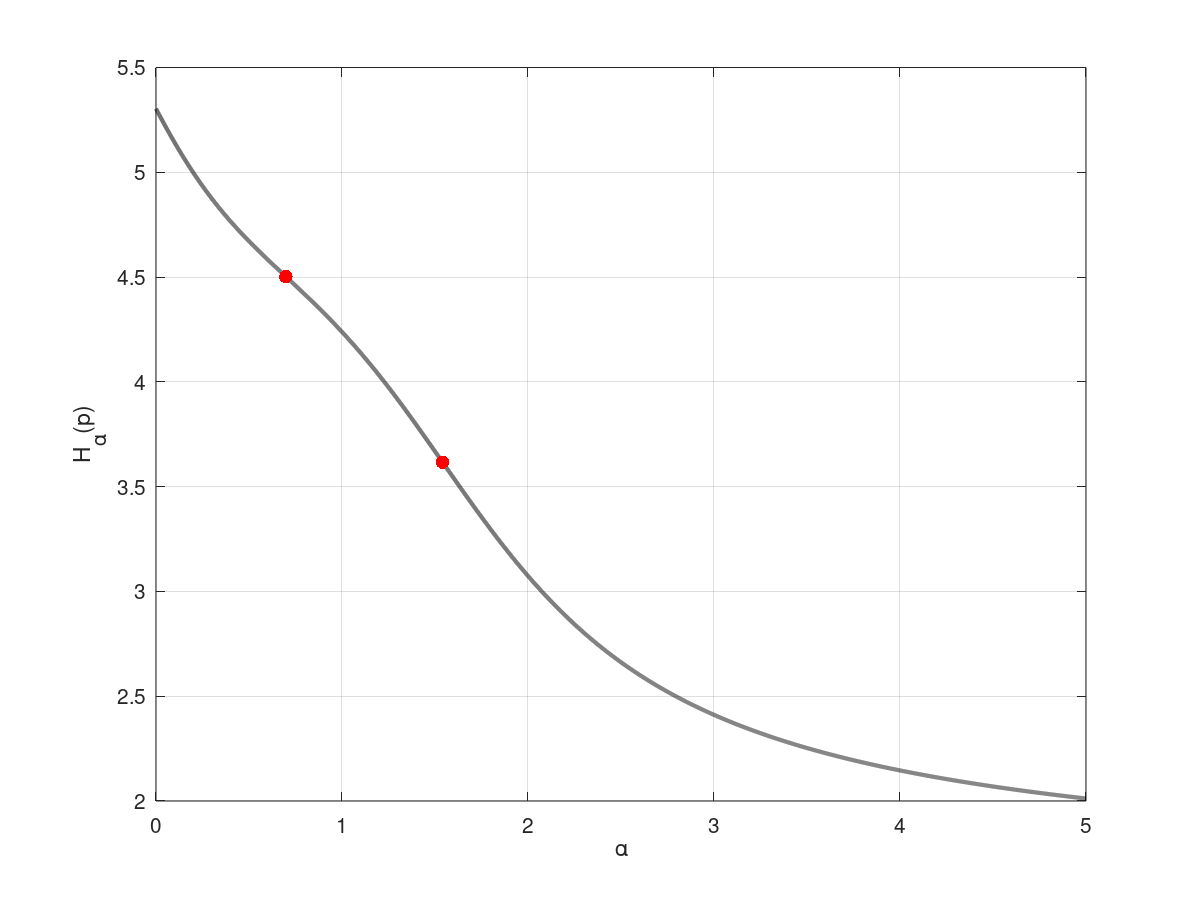}\label{fig:2infl}
\end{subfigure}%
\begin{subfigure}[h]{0.45\linewidth}
\includegraphics[width=1\linewidth]{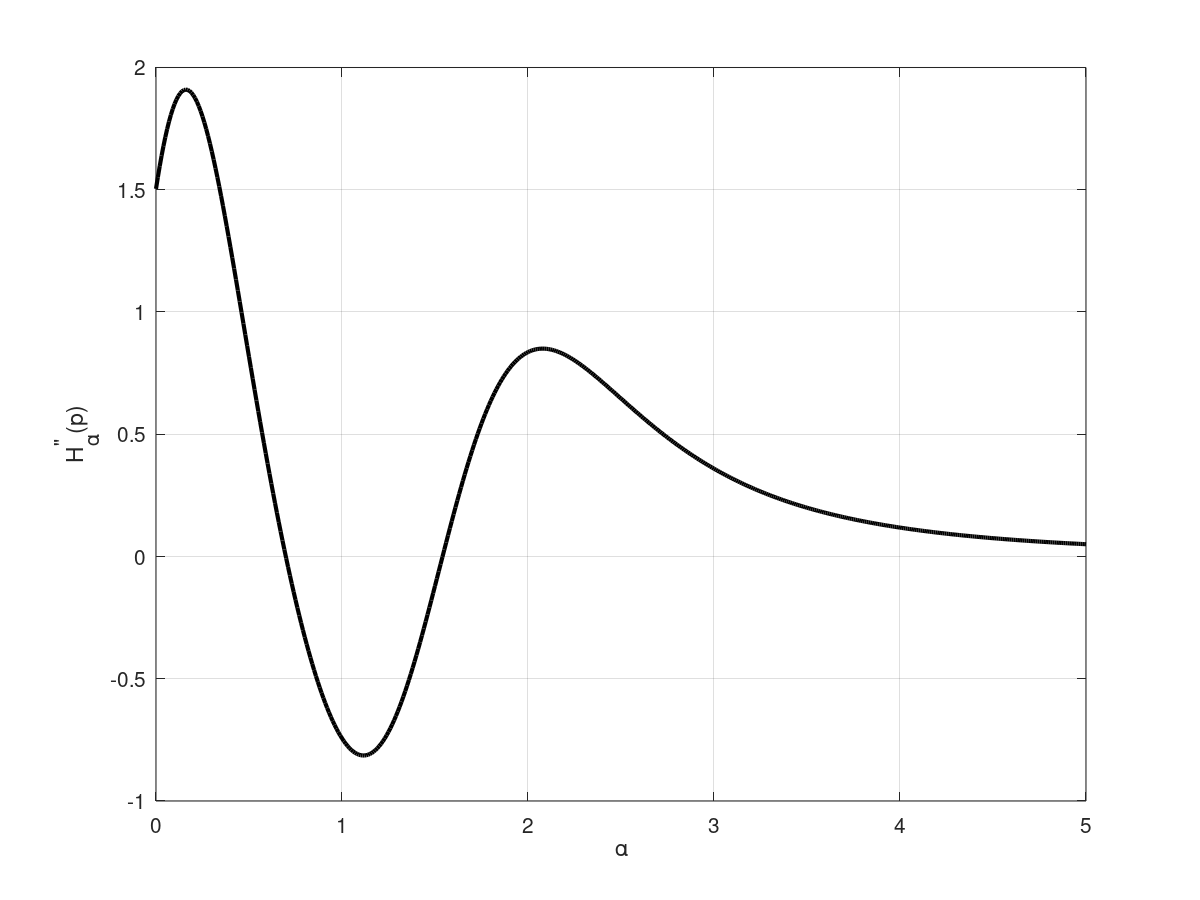}\label{fig:2infl_der}
\end{subfigure}
\caption{Graph of $\mathcal{H}_\alpha(p)$ and $\mathcal{H}^{''}_{\alpha}(p)$, where $p_1=...=p_{100}=0.0001, p_{101}=...=p_{200}=0.0079, p_{201}=0.2$. Here second derivative has two zeros.}
\end{center}
\end{figure}

\begin{figure}[h]
\begin{center}
\begin{minipage}[h]{0.6\linewidth}
\includegraphics[width=1\linewidth]{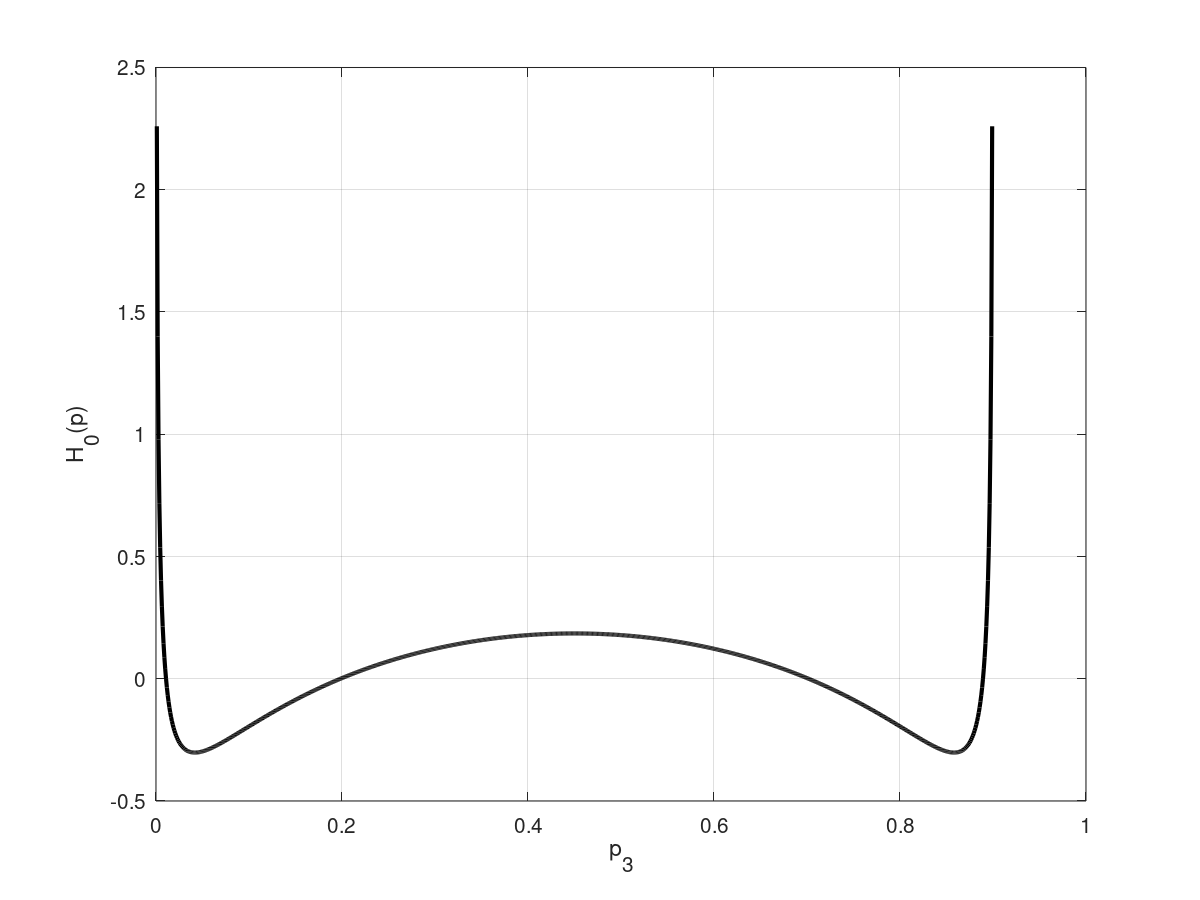}\label{fig:2fixed}
\caption{Graph of $\mathcal{H}_\alpha(p)$, where $p_1=p_2=0.05, p_3$ is changing from 0 to 0.9 and $p_4=1-p_1-p_2-p_3$.}
\end{minipage}
\end{center}
\end{figure}

\clearpage
\section{ Robustness of the   R\'{e}nyi entropy}\label{sec_3}
Now we study the asymptotic behavior of the   R\'{e}nyi entropy depending on the behavior of the involved probabilities.The first problem is the stability of the entropy w.r.t. involved probabilities and the rate of its convergence to the limit value when probabilities tend to their limit value with the fixed rate.
\subsection{Rate of convergence of the disturbed  entropy when the initial distribution is arbitrary but fixed}
Let's look at distributions that are ``near'' some fixed distribution $p=(p_k, \; 1\leq k \leq N)$ and   construct     the approximate distribution $p(\epsilon)=(p_k(\epsilon), \; 1\leq k \leq N)$ as follows. Now we can assume that some probabilities are zero, and we shall  see that this assumption influences the rate of convergence of the   R\'{e}nyi entropy to the limit value. So, let $0\leq N_1< N$ be a number of zero probabilities, and for them we consider approximate values of the form    $p_k(\epsilon)=c_k\varepsilon, \; 0\leq c_k\leq 1,\; 1\leq k\leq N_1$. Further, let $N_2=N-N_1\geq1$ be a number of non-zero probabilities, and for them we consider approximate values of the form  $p_k(\epsilon)=p_k+c_k\varepsilon, \; |c_k| \leq 1,\; N_1+1\leq k\leq N$, where $c_1+...+c_N=0$ and $ \varepsilon\leq\min\limits_{N_1+1\leq k\leq N}p_k$. Assume also  that there exists $  k\leq N$ such that  $c_k\neq0$, otherwise $  \mathcal {H}_\alpha(p)- \mathcal {H}_\alpha(p(\epsilon))=0$. So, we disturb intial probabilities linearly in $\varepsilon$ with different weights whose sum should necessarily be zero. These assumptions supply that $0\le p_k(\epsilon)\le 1$ and $ p_1(\epsilon)+...+p_N(\epsilon)=1$. Now we want to find out how entropy of the disturbed distribution will differ from the initial entropy, depending on parameters $\varepsilon, N$ and $\alpha$. We start with $\alpha=1.$

\begin{theorem}\label{epstozero1}
Let number $N$ and coefficients  $c_1,...,c_N$ be fixed, and let $\alpha=1$. We have three  different situations:
\begin{itemize}
\item[$(i)$] Let $N_1\geq1$ and there  exists $k\leq N_1$ such that $c_k\neq0.$ Then $$\mathcal{H}_1(p)-\mathcal{H}_1(p(\epsilon)) \sim \varepsilon\log\varepsilon\sum_{k=1}^{N_1} c_k, \; \varepsilon\to0.$$
\item[$(ii)$] Let for all $k\leq N_1 \; c_k=0$ and $\sum_{k=N_1+1}^{N} c_k\log p_k\neq0.$ Then $$\mathcal{H}_1(p)-\mathcal {H}_1(p(\epsilon)) \sim \varepsilon\sum_{k=N_1+1}^{N} c_k\log p_k, \; \varepsilon\to0.$$
\item[$(iii)$] Let for all $k\leq N_1 \; c_k=0$ and $\sum_{k=N_1+1}^{N} c_k\log p_k=0.$ Then $$\mathcal{H}_1(p)-\mathcal {H}_1(p(\epsilon)) \sim \frac{\varepsilon^2}{2}\sum_{k=N_1+1}^{N} \frac{c_k^2}{p_k}, \; \varepsilon\to0.$$
\end{itemize}
\end{theorem}

\begin{proof}
First of all, we will find asymptotic behavior of two auxiliary functions as $\varepsilon \to 0$.
First, let $  0\leq c_k\leq1$. Then
$$c_k\varepsilon\log(c_k\varepsilon) = c_k\varepsilon\log\varepsilon + c_k\varepsilon\log c_k = c_k\varepsilon\log\varepsilon + o(\varepsilon\log\varepsilon), \; \varepsilon \to 0.$$
Second, let $p_k>0, \; |c_k|\leq1$. Taking into account Taylor expansion of logarithm
\begin{align*}
\log(1+x)=x-\frac{x^2}{2}+o(x^2), \; x\to 0,
\end{align*}
we can write:
\begin{equation}\begin{gathered}\label{eq_4.}
(p_k+c_k\varepsilon)\log(p_k+c_k\varepsilon)  - p_k \log p_k= c_k\varepsilon\log p_k +  (p_k+c_k\varepsilon)\log(1+c_kp_k^{-1}\varepsilon)
\\  =
c_k\varepsilon\log  p_k + (p_k+c_k\varepsilon)\left(c_kp_k^{-1}\varepsilon - \frac{1}{2}(c_kp_k^{-1}\varepsilon)^2 + o(\varepsilon^2)\right)
\\  = \varepsilon(c_k\log p_k + c_k)+ \varepsilon^2\left(c_k^2p_k^{-1}-\frac12c_k^2p_k^{-1}\right)+o(\varepsilon^2)
\\  =\varepsilon(c_k\log p_k + c_k)+ \frac{c_k^2\varepsilon^2}{2p_k} +o(\varepsilon^2), \; \varepsilon \to 0.
\end{gathered}\end{equation}
In particular, we immediately  get from \eqref{eq_4.} that
\begin{align*}
(p_k+c_k\varepsilon)\log(p_k+c_k\varepsilon) - p_k \log p_k= o(\varepsilon\log\varepsilon), \; \varepsilon\to0,
\end{align*}
and
\begin{align*}
(p_k+c_k\varepsilon)\log(p_k+c_k\varepsilon) - p_k \log p_k=
\varepsilon(c_k\log p_k + c_k)+ o(\varepsilon), \; \varepsilon\to0
\end{align*}
Now simply observe the following.
\begin{align*}
(i) \lim_{\varepsilon\to0}\frac{\mathcal{H}_{1}(p)-\mathcal{H}_{1}(p(\epsilon))}{\varepsilon\log\varepsilon}&=
\lim_{\varepsilon\to0}\frac{1}{\varepsilon\log\varepsilon}\sum_{k=1}^{N_1}c_k\varepsilon\log c_k\varepsilon
\\ &+
\frac{1}{\varepsilon\log\varepsilon}\sum_{k=N_1+1}^{N} ((p_k+c_k\varepsilon)\log(p_k+c_k\varepsilon) - p_k \log p_k)
\\ &=
\lim_{\varepsilon\to0}\frac{1}{\varepsilon\log\varepsilon}\left(\sum_{k=1}^{N_1} (c_k\varepsilon\log\varepsilon+o(\varepsilon\log\varepsilon)) + \sum_{k=N_1+1}^{N} o(\varepsilon\log\varepsilon)\right)
\\ &=
\sum_{k=1}^{N_1} c_k.
\end{align*}
$(ii)$ Since for any   $  k\leq N_1$ we have that $c_k=0$ and the total sum $c_1+...+c_N=0$ then $c_{N_1+1}+...+c_N=0$. Furthermore, in this case
\begin{align*}
 \lim_{\varepsilon\to0}\frac{\mathcal{H}_{1}(p)-\mathcal{H}_{1}(p(\epsilon))}{\varepsilon}&=
\lim_{\varepsilon\to0}\frac{1}{\varepsilon}\sum_{k=N_1+1}^{N}((p_k+c_k\varepsilon)\log(p_k+c_k\varepsilon) - p_k \log p_k)
\\ &=
\lim_{\varepsilon\to0}\frac{1}{\varepsilon}\sum_{k=N_1+1}^{N}(\varepsilon(c_k\log p_k + c_k)+o(\varepsilon))
\\ &=
\sum_{k=N_1+1}^{N}(c_k\log p_k + c_k) = \sum_{k=N_1+1}^{N}c_k\log p_k.
\end{align*}
$(iii)$ In this case we have the following relations:
\begin{align*}
 \lim_{\varepsilon\to0}\frac{\mathcal{H}_{1}(p)-\mathcal{H}_{1}(p(\epsilon))}{\varepsilon^2}&=
\lim_{\varepsilon\to0}\frac{1}{\varepsilon^2}\sum_{k=N_1+1}^{N}((p_k+c_k\varepsilon)\log(p_k+c_k\varepsilon) - p_k \log p_k)
\\ &=
\lim_{\varepsilon\to0}\frac{1}{\varepsilon^2}\sum_{k=N_1+1}^{N}(\varepsilon(c_k\log p_k + c_k)+ \frac{c_k^2\varepsilon^2}{2p_k} +o(\varepsilon^2))
\\ &=
\lim_{\varepsilon\to0}\frac{1}{\varepsilon^2}\sum_{k=N_1+1}^{N}\left(\frac{c_k^2\varepsilon^2}{2p_k} +o(\varepsilon^2)\right)
= \frac12\sum_{k=N_1+1}^{N}\frac{c_k^2}{p_k}.
\end{align*}
Theorem is proved.
\end{proof}
Now we proceed with $\alpha<1$.
\begin{theorem}\label{epstozeroless1}
Let number $N$ and coefficients $c_1,...,c_N$ be fixed, and let $\alpha<1$. Then we have three  different situations:
\begin{itemize}
\item[$(i)$] Let $N_1\geq1$ and there exists $k\leq N_1$ such that $c_k\neq0.$ Then $$\mathcal{H}_\alpha(p)-\mathcal{H}_\alpha(p(\epsilon)) \sim \frac{\varepsilon^\alpha}{\alpha-1} \left(\sum_{k=1}^{N_1} c_k^\alpha\right) \left(\sum_{k=N_1+1}^{N} p_k^\alpha\right)^{-1}, \; \varepsilon\to0.$$
\item[$(ii)$] Let for all $k\leq N_1 \; c_k=0$ and $\sum_{k=N_1+1}^{N} c_k p_k^{\alpha-1}\neq0.$ Then $$\mathcal{H}_\alpha(p)-\mathcal{H}_\alpha(p(\epsilon)) \sim \frac{\alpha\varepsilon}{\alpha-1} \left(\sum_{k=N_1+1}^{N} c_k p_k^{\alpha-1}\right)\left(\sum_{k=N_1+1}^{N} p_k^\alpha\right)^{-1} , \; \varepsilon\to0.$$
\item[$(iii)$] Let for all $k\leq N_1 \; c_k=0$ and $\sum_{k=N_1+1}^{N} c_k p_k^{\alpha-1}=0.$ Then $$\mathcal{H}_\alpha(p)-\mathcal{H}_\alpha(p(\epsilon)) \sim \frac{\alpha\varepsilon^2}{2}\left(\sum_{k=N_1+1}^{N} c_k^2 p_k^{\alpha-2}\right)\left(\sum_{k=N_1+1}^{N} p_k^\alpha\right)^{-1}, \; \varepsilon\to0.$$
\end{itemize}
\end{theorem}

\begin{proof}
Similarly to proof of Theorem \ref{epstozero1}, we start with several  asymptotic relations  as $\varepsilon \to 0$.
Namely, let $p_k>0, \; |c_k|\leq1$. Taking into account Taylor expansion of $(1+x)^{\alpha}$ that has a form
\begin{align*}
(1+x)^{\alpha}=1+\alpha x +o(x), \; x\to 0,
\end{align*}
we can write:
\begin{equation}\begin{gathered}\label{equ_5}
\alpha c_k(p_k+c_k\varepsilon)^{\alpha-1} = \alpha c_k p_k^{\alpha-1} (1+c_k p_k^{-1}\varepsilon)^{\alpha-1}
\\  =
\alpha c_k p_k^{\alpha-1}(1+(\alpha-1)c_k p_k^{-1}\varepsilon + o(\varepsilon))
\\  =
\alpha c_k p_k^{\alpha-1} + \alpha(\alpha-1)c_k^2 p_k^{\alpha-2}\varepsilon + o(\varepsilon), \; \varepsilon \to 0.
 \end{gathered}\end{equation}
As a consequence, we   get the following asymptotic relations:
\begin{equation}\label{equ_6}
\alpha c_k(p_k+c_k\varepsilon)^{\alpha-1}= o(\varepsilon^{\alpha-1}), \; \varepsilon\to0,
\end{equation}
and
\begin{align*}
\alpha c_k(p_k+c_k\varepsilon)^{\alpha-1}= \alpha c_k p_k^{\alpha-1} + o(1), \; \varepsilon\to0
\end{align*}
$(i)$ Applying L’Hospital’s rule, we get:
\begin{align*}
 \lim_{\varepsilon\to0}\frac{H_{\alpha}(p)-H_{\alpha}(p(\epsilon))} {\varepsilon^\alpha}&=
\lim_{\varepsilon\to0}\frac{1}{(\alpha-1)\varepsilon^\alpha}
\log\left(\sum_{k=1}^{N_1}(c_k\varepsilon)^\alpha+\sum_{k=N_1+1}^{N} (p_k+c_k\varepsilon)^\alpha\right)
\\ &-
\frac{1}{(\alpha-1)\varepsilon^\alpha}\log\left(\sum_{k=1}^{N}p_k^\alpha\right)=
\frac{1}{\alpha-1}\lim_{\varepsilon\to0}\frac{1}{\alpha\varepsilon^{\alpha-1}}
\\ &\times
\frac{\sum_{k=1}^{N_1}\alpha c_k^\alpha\varepsilon^{\alpha-1}+\sum_{k=N_1+1}^{N} \alpha c_k(p_k+c_k\varepsilon)^{\alpha-1}} {\sum_{k=1}^{N_1}(c_k\varepsilon)^\alpha+\sum_{k=N_1+1}^{N} (p_k+c_k\varepsilon)^\alpha}
\\ &=
\frac{1}{\alpha-1}\lim_{\varepsilon\to0}
\frac{\sum_{k=1}^{N_1} c_k^\alpha+\varepsilon^{1-\alpha}\sum_{k=N_1+1}^{N} o(\varepsilon^{\alpha-1})} {\sum_{k=1}^{N_1}(c_k\varepsilon)^\alpha+\sum_{k=N_1+1}^{N} (p_k+c_k\varepsilon)^\alpha}
\\ &=
\frac{1}{\alpha-1}\left(\sum_{k=1}^{N_1} c_k^\alpha\right) \left(\sum_{k=N_1+1}^{N} p_k^\alpha\right)^{-1}.
\end{align*}
$(ii)$ In this case we can transform  the value under a limit as follows:
\begin{align*}
 \lim_{\varepsilon\to0}\frac{\mathcal{H}_{\alpha}(p)-\mathcal{H}_{\alpha}(p(\epsilon))}{\varepsilon}&=
\lim_{\varepsilon\to0}\frac{1}{(\alpha-1)\varepsilon} \left(\log\left(\sum_{k=N_1+1}^{N} (p_k+c_k\varepsilon)^\alpha\right) -\log\left(\sum_{k=1}^{N}p_k^\alpha\right)\right)
\\ &=
\frac{1}{\alpha-1}\lim_{\varepsilon\to0}\frac{\sum_{k=N_1+1}^{N} \alpha c_k(p_k+c_k\varepsilon)^{\alpha-1}}{\sum_{k=N_1+1}^{N} (p_k+c_k\varepsilon)^\alpha}
\\ &=
\frac{1}{\alpha-1}\lim_{\varepsilon\to0}\frac{\sum_{k=N_1+1}^{N} (\alpha c_k p_k^{\alpha-1} + o(1))}{\sum_{k=N_1+1}^{N} (p_k+c_k\varepsilon)^\alpha}
\\ &=
\frac{\alpha}{\alpha-1}\left(\sum_{k=N_1+1}^{N} c_k p_k^{\alpha-1}\right) \left(\sum_{k=N_1+1}^{N} p_k^\alpha\right)^{-1}.
\end{align*}
$(iii)$ Finally, in the 3rd case,
\begin{align*}
 \lim_{\varepsilon\to0}\frac{\mathcal{H}_{\alpha}(p)-\mathcal{H}_{\alpha}(p(\epsilon))}{\varepsilon^2}&=
\lim_{\varepsilon\to0}\frac{1}{(\alpha-1)\varepsilon^2} \left(\log\left(\sum_{k=N_1+1}^{N} (p_k+c_k\varepsilon)^\alpha\right) -\log\left(\sum_{k=1}^{N}p_k^\alpha\right)\right)
\\ &=
\frac{1}{\alpha-1}\lim_{\varepsilon\to0}\frac{1}{2\varepsilon}\frac{\sum_{k=N_1+1}^{N} \alpha c_k(p_k+c_k\varepsilon)^{\alpha-1}}{\sum_{k=N_1+1}^{N} (p_k+c_k\varepsilon)^\alpha}
\\ &=
\frac{1}{\alpha-1}\lim_{\varepsilon\to0}\frac{1}{2\varepsilon}\frac{\sum_{k=N_1+1}^{N} (\alpha c_k p_k^{\alpha-1} + \alpha(\alpha-1)c_k^2 p_k^{\alpha-2}\varepsilon + o(\varepsilon))}{\sum_{k=N_1+1}^{N} (p_k+c_k\varepsilon)^\alpha}
\\ &=
\frac{1}{\alpha-1}\lim_{\varepsilon\to0}\frac{1}{2\varepsilon}\frac{\sum_{k=N_1+1}^{N} (\alpha(\alpha-1)c_k^2 p_k^{\alpha-2}\varepsilon + o(\varepsilon))}{\sum_{k=N_1+1}^{N} (p_k+c_k\varepsilon)^\alpha}
\\ &=
\frac{\alpha}{2}\left(\sum_{k=N_1+1}^{N} c_k^2 p_k^{\alpha-2}\right) \left(\sum_{k=N_1+1}^{N} p_k^\alpha\right)^{-1}.
\end{align*}
Theorem is proved.
\end{proof}
Now we conclude with $\alpha>1$. In this case, five different asymptotics are possible.
\begin{theorem}\label{epstozeromore1}
Let number  $N$ and coefficients $c_1,...,c_N$ be fixed,  and let $\alpha>1$. Then     five  different situations are possible:
\begin{itemize}
\item[$(i)$] Let $\sum_{k=N_1+1}^{N} c_k p_k^{\alpha-1}\neq0.$ Then whatever $N_1\geq 0$ and $\alpha>1$ are equal,   we have that $$\mathcal{H}_\alpha(p)-\mathcal{H}_\alpha(p(\epsilon)) \sim \frac{\alpha\varepsilon}{\alpha-1} \left(\sum_{k=N_1+1}^{N} c_k p_k^{\alpha-1}\right)\left(\sum_{k=N_1+1}^{N} p_k^\alpha\right)^{-1} , \; \varepsilon\to0.$$
\item[$(ii)$] Let  $\sum_{k=N_1+1}^{N} c_k p_k^{\alpha-1}=0$, $N_1\geq1$, and there exists $k\leq N_1$ such that $c_k\neq0$. Then for $\alpha<2 $ it holds that $$\mathcal{H}_\alpha(p)-\mathcal{H}_\alpha(p(\epsilon)) \sim \frac{\varepsilon^\alpha}{\alpha-1} \left(\sum_{k=1}^{N_1} c_k^\alpha\right) \left(\sum_{k=N_1+1}^{N} p_k^\alpha\right)^{-1} , \; \varepsilon\to0.$$
\item[$(iii)$] Let  $\sum_{k=N_1+1}^{N} c_k p_k^{\alpha-1}=0$, $N_1\ge 0$ and for all $k\leq N_1$ we have that $c_k=0$.     Then for $\alpha<2 $ it holds that $$\mathcal{H}_\alpha(p)-\mathcal{H}_\alpha(p(\epsilon)) \sim \frac{\alpha\varepsilon^2}{2} \left(\sum_{k=N_1+1}^{N} c_k^2 p_k^{\alpha-2}\right) \left(\sum_{k=N_1+1}^{N} p_k^\alpha\right)^{-1} , \; \varepsilon\to0.$$
\item[$(iv)$] Let $\sum_{k=N_1+1}^{N} c_k p_k^{\alpha-1}=0, \; \alpha=2.$ Then whatever $N_1\geq 0$ and $c_k$ for $k\leq N_1$ are equal,   we have that $$\mathcal{H}_\alpha(p)-\mathcal{H}_\alpha(p(\epsilon)) \sim \varepsilon^2 \left(\sum_{k=1}^{N} c_k^2\right) \left(\sum_{k=N_1+1}^{N} p_k^2\right)^{-1} , \; \varepsilon\to0.$$
\item[$(v)$] Let $\sum_{k=N_1+1}^{N} c_k p_k^{\alpha-1}=0, \; \alpha>2.$ Then whatever $N_1\geq 0$ and $c_k$ for $k\leq N_1$ are equal,    we have that $$\mathcal{H}_\alpha(p)-\mathcal{H}_\alpha(p(\epsilon)) \sim \frac{\alpha\varepsilon^2}{2} \left(\sum_{k=N_1+1}^{N} c_k^2 p_k^{\alpha-2}\right) \left(\sum_{k=N_1+1}^{N} p_k^\alpha\right)^{-1} , \; \varepsilon\to0.$$
\end{itemize}
\end{theorem}

\begin{proof} As in the proof of Theorem \ref{epstozeroless1}, we shall use   expansions \eqref{equ_5} and \eqref{equ_6}. The main tool will be
  L’Hospital’s rule.

$(i)$  Let $\sum_{k=N_1+1}^{N} c_k p_k^{\alpha-1}\neq0.$ Then whatever $N_1\geq 0$ and $\alpha>1$ are equal, we have the following relations:
\begin{align*}
 \lim_{\varepsilon\to0}\frac{\mathcal{H}_{\alpha}(p)-\mathcal{H}_{\alpha}(p(\epsilon))} {\varepsilon}&=
\lim_{\varepsilon\to0}\frac{1}{(\alpha-1)\varepsilon}
\log\left(\sum_{k=1}^{N_1}(c_k\varepsilon)^\alpha+\sum_{k=N_1+1}^{N} (p_k+c_k\varepsilon)^\alpha\right)
\\ &-
\frac{1}{(\alpha-1)\varepsilon}\log\left(\sum_{k=1}^{N}p_k^\alpha\right)=
\frac{1}{\alpha-1}
\\ &\times
\lim_{\varepsilon\to0}\frac{\sum_{k=1}^{N_1}\alpha c_k^\alpha\varepsilon^{\alpha-1}+\sum_{k=N_1+1}^{N} \alpha c_k(p_k+c_k\varepsilon)^{\alpha-1}} {\sum_{k=1}^{N_1}(c_k\varepsilon)^\alpha+\sum_{k=N_1+1}^{N} (p_k+c_k\varepsilon)^\alpha}
\\ &=
\frac{\alpha}{\alpha-1}\left(\sum_{k=N_1+1}^{N} c_k p_k^{\alpha-1}\right) \left(\sum_{k=N_1+1}^{N} p_k^\alpha\right)^{-1}.
\end{align*}
$(ii)$ Let  $\sum_{k=N_1+1}^{N} c_k p_k^{\alpha-1}=0$, $N_1\geq1$, and there exists $k\leq N_1$ such that $c_k\neq0$. Then for $\alpha<2 $ we have that
\begin{align*}
 \lim_{\varepsilon\to0}\frac{\mathcal{H}_{\alpha}(p)-\mathcal{H}_{\alpha}(p(\epsilon))} {\varepsilon^\alpha}&=
\lim_{\varepsilon\to0}\frac{1}{(\alpha-1)\varepsilon^\alpha}
\log\left(\sum_{k=1}^{N_1}(c_k\varepsilon)^\alpha+\sum_{k=N_1+1}^{N} (p_k+c_k\varepsilon)^\alpha\right)
\\ &-
\frac{1}{(\alpha-1)\varepsilon^\alpha}\log\left(\sum_{k=1}^{N}p_k^\alpha\right)=
\frac{1}{\alpha-1}\lim_{\varepsilon\to0}\frac{1}{\alpha\varepsilon^{\alpha-1}}
\\ &\times
\frac{\sum_{k=1}^{N_1}\alpha c_k^\alpha\varepsilon^{\alpha-1}+\sum_{k=N_1+1}^{N} \alpha c_k(p_k+c_k\varepsilon)^{\alpha-1}} {\sum_{k=1}^{N_1}(c_k\varepsilon)^\alpha+\sum_{k=N_1+1}^{N} (p_k+c_k\varepsilon)^\alpha}
\\ &=
\frac{1}{\alpha-1}\lim_{\varepsilon\to0}
\frac{\sum_{k=1}^{N_1} c_k^\alpha\varepsilon^{\alpha-1}+\sum_{k=N_1+1}^{N} ((\alpha-1) c_k^2 p_k^{\alpha-2}\varepsilon + o(\varepsilon))} {\varepsilon^{\alpha-1}\left(\sum_{k=1}^{N_1}(c_k\varepsilon)^\alpha+\sum_{k=N_1+1}^{N} (p_k+c_k\varepsilon)^\alpha\right)}
\\ &=
\frac{1}{\alpha-1} \left(\sum_{k=1}^{N_1} c_k^\alpha\right) \left(\sum_{k=N_1+1}^{N} p_k^\alpha\right)^{-1}.
\end{align*}
$(iii)$  Let  $\sum_{k=N_1+1}^{N} c_k p_k^{\alpha-1}=0$, $N_1\ge 0$ and for all $k\leq N_1$ we have that $c_k=0$.     Then for $\alpha<2 $ it holds that
\begin{align*}
 \lim_{\varepsilon\to0}\frac{\mathcal{H}_{\alpha}(p)-\mathcal{H}_{\alpha}(p(\epsilon))}{\varepsilon^2}&=
\lim_{\varepsilon\to0}\frac{1}{(\alpha-1)\varepsilon^2} \left(\log\left(\sum_{k=N_1+1}^{N} (p_k+c_k\varepsilon)^\alpha\right) -\log\left(\sum_{k=1}^{N}p_k^\alpha\right)\right)
\\ &=
\frac{1}{\alpha-1}\lim_{\varepsilon\to0}\frac{1}{2\varepsilon}\frac{\sum_{k=N_1+1}^{N} \alpha c_k(p_k+c_k\varepsilon)^{\alpha-1}}{\sum_{k=N_1+1}^{N} (p_k+c_k\varepsilon)^\alpha}
\\ &=
\frac{1}{\alpha-1}\lim_{\varepsilon\to0}\frac{\sum_{k=N_1+1}^{N} (\alpha c_k p_k^{\alpha-1} + \alpha(\alpha-1)c_k^2 p_k^{\alpha-2}\varepsilon + o(\varepsilon))}{2\varepsilon\left(\sum_{k=N_1+1}^{N} (p_k+c_k\varepsilon)^\alpha\right)}
\\ &=
\frac{1}{\alpha-1}\lim_{\varepsilon\to0}\frac{\sum_{k=N_1+1}^{N} (\alpha(\alpha-1)c_k^2 p_k^{\alpha-2}\varepsilon + o(\varepsilon))}{2\varepsilon\left(\sum_{k=N_1+1}^{N} (p_k+c_k\varepsilon)^\alpha\right)}
\\ &=
\frac{\alpha}{2} \left(\sum_{k=N_1+1}^{N} c_k^2 p_k^{\alpha-2}\right) \left(\sum_{k=N_1+1}^{N} p_k^\alpha\right)^{-1}.
\end{align*}
$(iv)$ Obviously, in the case $\alpha=2$ we have the simple value of the entropy: $$\mathcal{H}_2(p)=-\log\left(\sum_{k=1}^{N} p_k^2\right).$$
Therefore, if $\sum_{k=N_1+1}^{N} c_k p_k^{\alpha-1}=0, \; \alpha=2,$ then, whatever $N_1\geq 0$ and $c_k$ for $k\leq N_1$ are equal,   we have that
\begin{align*}
 \lim_{\varepsilon\to0}\frac{\mathcal{H}_2(p)-\mathcal{H}_2(p(\epsilon))}{\varepsilon^2}&=
\lim_{\varepsilon\to0}\frac{1}{\varepsilon^2} \log\left(\sum_{k=1}^{N_1} (c_k\varepsilon)^2 + \sum_{k=N_1+1}^{N} (p_k+c_k\varepsilon)^2\right)
\\ &-
\frac{1}{\varepsilon^2}\log\left(\sum_{k=1}^{N}p_k^2\right)=
\lim_{\varepsilon\to0}\frac{1}{2\varepsilon}
\\ &\times
\frac{\sum_{k=1}^{N_1} 2c_k^2\varepsilon + \sum_{k=N_1+1}^{N} 2c_k(p_k+c_k\varepsilon)} {\sum_{k=1}^{N_1} (c_k\varepsilon)^2 + \sum_{k=N_1+1}^{N} (p_k+c_k\varepsilon)^2}
\\ &=
\lim_{\varepsilon\to0} \frac{\sum_{k=1}^{N_1} c_k^2 + \sum_{k=N_1+1}^{N} c_k^2} {\sum_{k=1}^{N_1} (c_k\varepsilon)^2 + \sum_{k=N_1+1}^{N} (p_k+c_k\varepsilon)^2}
\\ &=
\left(\sum_{k=1}^{N} c_k^2\right) \left(\sum_{k=N_1+1}^{N} p_k^2\right)^{-1}.
\end{align*}
$(v)$ Let $\sum_{k=N_1+1}^{N} c_k p_k^{\alpha-1}=0, \; \alpha>2.$ Then whatever $N_1\geq 0$ and $c_k$ for $k\leq N_1$ are equal,    we have that
\begin{align*}
 \lim_{\varepsilon\to0}\frac{\mathcal{H}_{\alpha}(p)-\mathcal{H}_{\alpha}(p(\epsilon))} {\varepsilon^2}&=
\lim_{\varepsilon\to0}\frac{1}{(\alpha-1)\varepsilon^2}
\log\left(\sum_{k=1}^{N_1}(c_k\varepsilon)^\alpha+\sum_{k=N_1+1}^{N} (p_k+c_k\varepsilon)^\alpha\right)
\\ &-
\frac{1}{(\alpha-1)\varepsilon^2}\log\left(\sum_{k=1}^{N}p_k^\alpha\right)=
\frac{1}{\alpha-1}\lim_{\varepsilon\to0}\frac{1}{2\varepsilon}
\\ &\times
\frac{\sum_{k=1}^{N_1}\alpha c_k^\alpha\varepsilon^{\alpha-1}+\sum_{k=N_1+1}^{N} \alpha c_k(p_k+c_k\varepsilon)^{\alpha-1}} {\sum_{k=1}^{N_1}(c_k\varepsilon)^\alpha+\sum_{k=N_1+1}^{N} (p_k+c_k\varepsilon)^\alpha}
\\ &=
\frac{1}{\alpha-1}\lim_{\varepsilon\to0}
\frac{\sum_{k=1}^{N_1} \alpha c_k^\alpha\varepsilon^{\alpha-1}+\sum_{k=N_1+1}^{N} (\alpha(\alpha-1) c_k^2 p_k^{\alpha-2}\varepsilon + o(\varepsilon))} {2\varepsilon\left(\sum_{k=1}^{N_1}(c_k\varepsilon)^\alpha+\sum_{k=N_1+1}^{N} (p_k+c_k\varepsilon)^\alpha\right)}
\\ &=
\frac{\alpha}{2} \left(\sum_{k=N_1+1}^{N} c_k^2 p_k^{\alpha-2}\right) \left(\sum_{k=N_1+1}^{N} p_k^\alpha\right)^{-1}.
\end{align*}
Theorem is proved.
\end{proof}

\clearpage
\subsection{Convergence of the disturbed entropy when the initial distribution is uniform but the number of events increases to $\infty$} The second problem is to establish conditions of stability of the entropy of uniform distribution when the number of events tends to $\infty$.
Let $N>1$, $p_N(uni)=(\frac{1}{N},\ldots,\frac{1}{N})$ be a vector of uniform $N$-dimensional distribution, $\varepsilon=\varepsilon(N)\leq\frac1N$, and
$\{c_{kN}; \; N\geq1, \; 1\leq k\leq N\}$ be a family of fixed numbers (not totally zero) such that  $|c_{kN}|\leq 1$ and $\sum_{k=1}^{N}c_{kN}=0$. Note that for any $N\ge 1$ there are   strictly positive numbers $c_{kN}$ for some $k$ and consider the disturbed distribution vector $p_N^{'}=(\frac1N+c_{1N}\varepsilon,...,\frac1N+c_{NN}\varepsilon)$.

\begin{theorem}\label{Ntoinf}
Let $\varepsilon(N)=o(\frac1N),\;  N \to \infty$. Then $$\mathcal{H}_{\alpha}(p_N)-\mathcal{H}_{\alpha}(p_N^{'}) \to 0, \; N \to \infty.$$
\end{theorem}

\begin{proof}
We know that $N\varepsilon \to 0, \; N \to \infty$ and the family of numbers $\{c_{kn}; \; n \geq 1, \; 1\leq k\leq n\}$ is bounded. Therefore the values $$\sup\limits_{n \geq 1, \; 1\leq k\leq n} (1+Nc_{kn}\varepsilon)\to1, \; \inf\limits_{n \geq 1, \; 1\leq k\leq n} (1+Nc_{kn}\varepsilon)\to1, \; N \to \infty$$ as the function of $N$, and for every $N\ge 1 \; \sup\limits_{n \geq 1, \; 1\leq k\leq n} (1+Nc_{kn}\varepsilon)\geq1.$
Recall  that function $x\log x$ is increasing in $x\geq 1$ and $x\log x\leq0$ for $0<x<1$. Moreover, Renyi entropy is maximal on the uniform distribution. As a consequence of all these observations and assumptions  we get that
\begin{align*}
0\leq \mathcal{H}_1(p_N)-\mathcal{H}_1(p_N^{'}) &= \frac1N\sum_{k=1}^{N} (1 + Nc_{kN}\varepsilon)\log(1 + Nc_{kN}\varepsilon)
\\ &\leq
\frac1N\sum_{k=1}^{N}\sup\limits_{n \geq 1, \; 1\leq k\leq n} (1+Nc_{kn}\varepsilon)\log\sup\limits_{n \geq 1, \; 1\leq k\leq n} (1+Nc_{kn}\varepsilon)
\\ &=
\sup\limits_{n \geq 1, \; 1\leq k\leq n} (1+Nc_{kn}\varepsilon)\log\sup\limits_{n \geq 1, \; 1\leq k\leq n} (1+Nc_{kn}\varepsilon)
\\ &\to
0, \; N \to \infty.
\end{align*}
Let $\alpha>1$. Then
\begin{align*}
0\leq \mathcal{H}_\alpha(p_N) - \mathcal{H}_\alpha(p_N^{'}) &= \frac{1}{\alpha-1}\log\left(\frac1N\sum_{k=1}^{N} (1+Nc_{kN}\varepsilon)^{\alpha}\right)
\\ &\leq
\frac{1}{\alpha-1}\log\left(\frac1N\sum_{k=1}^{N} \left(\sup\limits_{n \geq 1, \; 1\leq k\leq n} (1+Nc_{kn}\varepsilon)\right)^{\alpha}\right)
\\ &=
\frac{\alpha}{\alpha-1}\log\left(\sup\limits_{n \geq 1, \; 1\leq k\leq n} (1+Nc_{kn}\varepsilon)\right) \to0, \; N \to \infty.
\end{align*}
Similarly, for $0<\alpha<1$ we produce the transformations:
\begin{align*}
0\leq \mathcal{H}_\alpha(p_N) - \mathcal{H}_\alpha(p_N^{'}) &= \frac{1}{\alpha-1}\log\left(\frac1N\sum_{k=1}^{N} (1+Nc_{kN}\varepsilon)^{\alpha}\right)
\\ &\leq
\frac{1}{\alpha-1}\log\left(\frac1N\sum_{k=1}^{N} \left(\inf\limits_{n \geq 1, \; 1\leq k\leq n} (1+Nc_{kn}\varepsilon)\right)^{\alpha}\right)
\\ &=
\frac{\alpha}{\alpha-1}\log\left(\inf\limits_{n \geq 1, \; 1\leq k\leq n} (1+Nc_{kn}\varepsilon)\right) \to0, \; N \to \infty,
\end{align*}
and the proof follows.
\end{proof}

\subsection{Binomial and Poisson distribution}
In this section we look at convergence of R{\'e}nyi entropy of binomial distribution to R{\'e}nyi entropy of Poisson distribution.
\begin{theorem}\label{BintoPoi}
Let $\lambda>0$ be fixed. For any $\alpha>0$
$$\lim_{n\to\infty} \mathcal{H}_{\alpha}\left(B\left(n,\frac{\lambda}{n}\right)\right)= \mathcal{H}_{\alpha}(Poi(\lambda)).$$
\end{theorem}

\begin{proof}
First, let $\alpha=1$. We will find and regroup entropy of binomial and Poisson distribution.
\begin{align*}
\mathcal{H}_1\left(B\left(n,p\right)\right)&= -\sum_{k=0}^{n}{n \choose k}p^k(1-p)^{n-k} \log\left({n \choose k} p^k(1-p)^{n-k}\right)
\\ &=
-\sum_{k=0}^{n}{n \choose k}p^k(1-p)^{n-k}\log{n \choose k} -
n\left(p\log p+(1-p)\log(1-p)\right)
\\ &=
-\sum_{k=0}^{n}{n \choose k}p^k(1-p)^{n-k}(\log n! -\log k! - \log (n-k)!)+np\log n
\\ &-
np\log np - n\log(1-p)+np\log(1-p) = np\log(1-p)-n\log(1-p)
\\ &-
np\log np+ np\log n-\sum_{k=0}^{n}{n \choose k}p^k(1-p)^{n-k} (\log n! -\log k! - \log (n-k)!).
\end{align*}
\begin{align*}
\mathcal{H}_1(Poi(\lambda))=-\sum_{k=0}^{\infty}e^{-\lambda}\frac{\lambda^k}{k!}\log \left(e^{-\lambda}\frac{\lambda^k}{k!}\right)=\lambda-\lambda\log\lambda + \sum_{k=0}^{\infty}e^{-\lambda}\frac{\lambda^k}{k!}\log k!
\end{align*}
We want to show componentwise convergence of entropies. For that let's take $np=\lambda$ and observe that:
\begin{align*}
np\log(1-p)=\lambda\log(1-p)\to\lambda\log1=0, \; n \to \infty, \; p \to 0.
\end{align*}
\begin{align*}
-n\log(1-p)=\log\left(1-\frac{\lambda}{n}\right)^{-n}\to\lambda, \; n \to \infty, \; p \to 0.
\end{align*}
\begin{align*}
-np\log np=-\lambda\log\lambda.
\end{align*}
\begin{align*}
np\log n&-\sum_{k=0}^{n}{n \choose k}p^k(1-p)^{n-k}(\log n! -\log k! - \log (n-k)!)
\\ &=
\sum_{k=0}^{n}{n \choose k}p^k(1-p)^{n-k}k\log n-\sum_{k=0}^{n}{n \choose k}p^k(1-p)^{n-k}(\log n! -\log k! - \log (n-k)!)
\\ &=
\sum_{k=0}^{n}{n \choose k}p^k(1-p)^{n-k}\left(\log n^k - \log n! + \log k! + \log (n-k)!\right)
\\ &=
\sum_{k=0}^{n}{n \choose k}p^k(1-p)^{n-k}\left(\log \frac{n^k (n-k)!}{n!} + \log k! \right)
\end{align*}
It is well-known that $\frac{\log(x)}{x}\leq 1, \; x>0$. Using this fact, we  get the following representation
\begin{align*}
{n \choose k}p^k(1-p)^{n-k}\log \frac{n^k (n-k)!}{n!} &= \frac{n!}{(n-k)!k!} \left(\frac{\lambda}{n}\right)^k\left(1-\frac{\lambda}{n}\right)^{n-k}\log \frac{n^k (n-k)!}{n!}
\\ &=
\frac{\lambda^k}{k!}\left(1-\frac{\lambda}{n}\right)^{n-k}
\frac{n!}{n^k(n-k)!} \log \frac{n^k (n-k)!}{n!}
\\ &\leq
\frac{\lambda^k}{k!}\left(1-\frac{\lambda}{n}\right)^{n-k}
\leq \frac{\lambda^k}{k!}.
\end{align*}
For the second part of sum simply observe that:
\begin{align*}
{n \choose k}p^k(1-p)^{n-k} \log k! &= \frac{n!}{(n-k)!k!} \left(\frac{\lambda}{n}\right)^k\left(1-\frac{\lambda}{n}\right)^{n-k} \log k!
\\ &=
\frac{\lambda^k}{k!}\log k! \left(1-\frac{\lambda}{n}\right)^{n-k} \frac{n!}{(n-k)!n^k}
\\ &\leq
\frac{\lambda^k}{k!}\log k!
\end{align*}
$\sum_{k=0}^{\infty}\frac{\lambda^k}{k!}\left(1+\log k!\right)<\infty$, thus, by Lebesgue's dominated convergence theorem:
\begin{align*}
\lim_{n\to\infty}&\sum_{k=0}^{n}{n \choose k}p^k(1-p)^{n-k}\left(\log\frac{n^k(n-k)!}{n!}+\log k!\right)
\\ &=
\sum_{k=0}^{\infty}\lim_{n\to\infty}{n \choose k}p^k(1-p)^{n-k}\left(\log\frac{n^k(n-k)!}{n!}+\log k!\right)
\\ &=
\sum_{k=0}^{\infty}\lim_{n\to\infty}\frac{\lambda^k}{k!}\left(1-\frac{\lambda}{n}\right)^{n-k} \frac{n!}{(n-k)!n^k}\left(\log \frac{n^k (n-k)!}{n!}+ \log k!\right)
\\ &=
\sum_{k=0}^{\infty}e^{-\lambda}\frac{\lambda^k}{k!}\log k!
\end{align*}
Finally, we get that
$$\lim_{n\to\infty} \mathcal{H}_1\left(B\left(n,\frac{\lambda}{n}\right)\right)= \mathcal{H}_1(Poi(\lambda)).$$
For $\alpha\neq1$ we have:
$$\mathcal{H}_{\alpha}(binomial)=\frac{1}{1-\alpha}\log\sum_{k=0}^{n}\left({n \choose k}p^k(1-p)^{n-k}\right)^\alpha,$$
$$\mathcal{H}_{\alpha}(poisson)=\frac{1}{1-\alpha}\log\sum_{k=0}^{+\infty}\left(e^{-\lambda}\frac{\lambda^k}{k!}\right)^\alpha.$$
Thus, to show that
$$\lim_{n\to\infty} \mathcal{H}_{\alpha}\left(B\left(n,\frac{\lambda}{n}\right)\right)= \mathcal{H}_{\alpha}(Poi(\lambda)),$$
it is enough to show convergence of sums which follows from Lebesgue's dominated convergence theorem and
$$\left({n \choose k}p^k(1-p)^{n-k}\right)^\alpha\leq\left(\frac{\lambda^k}{k!}\right)^\alpha, \; \sum_{k=0}^{+\infty} \left(\frac{\lambda^k}{k!}\right)^\alpha < +\infty.$$
\end{proof}

\section{Appendix}\label{sec_4}
We  let $0\log0=0$ by continuity  and  prove several auxiliary results. Stating these three lemmas, we assume that $p_i\geq0, 1\le i\le N$ are fixed.
\begin{lemma}\label{entropyalpha1}
$H_\alpha(p) \to H(p), \; \alpha \to 1.$
\end{lemma}

\begin{proof}
Using L'Hospital's rule, we get the following relations:
\begin{align*}
\lim_{\alpha\to 1} H_\alpha(p) &= \lim_{\alpha\to 1} \frac{1}{1-\alpha}\log\left( \sum_{k=1}^{N} p_k^{\alpha}\right) = \lim_{\alpha\to 1} \frac{1}{-1} \frac{\sum_{k=1}^{N} p_k^{\alpha}\log(p_k)}{\sum_{k=1}^{N} p_k^{\alpha}}
\\ &=
-\sum_{k=1}^{N} p_k\log(p_k)= H(p).
\end{align*}
\end{proof}

\noindent Let $\mathcal{H}_1(p) := \mathcal{H}(p)$ (Shannon entropy), and  so   $H_\alpha(p)$ is defined for all $\alpha>0$ and is continuous in $\alpha$.

\begin{lemma}\label{entropydecr}
$H_\alpha(p)$ is non-increasing in $\alpha>0$.
\end{lemma}

\begin{proof} Indeed,
\begin{align*}
\frac{\partial H_\alpha(p)}{\partial \alpha} &= \frac{1}{(1-\alpha)^2} \log\left(\sum_{i=1}^{N} p_i^{\alpha}\right) + \frac{1}{1-\alpha} \frac {\sum_{k=1}^{N} p_k^{\alpha}\log p_k} {\sum_{k=1}^{N} p_k^{\alpha}}
\\ &=\frac{1}{(1-\alpha)^2 \sum_{k=1}^{N} p_k^{\alpha}}
\sum_{k=1}^{N} p_k^{\alpha}\left(\log\left(\sum_{i=1}^{N} p_i^{\alpha}\right) + \log p_k^{1-\alpha}\right)
\\ &=
\frac{-1}{(1-\alpha)^2}
\sum_{k=1}^{N} \frac{p_k^{\alpha}}{\sum_{i=1}^{N} p_i^{\alpha}} \log\left(\frac {p_k^\alpha} {\sum_{i=1}^{N} p_i^{\alpha}}\frac{1}{p_k}\right).
\end{align*}
Let $q_k=\frac{p_k^{\alpha}}{\sum_{i=1}^{N} p_i^{\alpha}}$. Then
\begin{align*}
\frac{\partial H_\alpha(p)}{\partial \alpha} = \frac{-1}{(1-\alpha)^2}
\sum_{k=1}^{N} q_k \log\frac{q_k}{p_k} \leq 0.
\end{align*}
The fact that $\mathcal{H}_\alpha(p) \leq \mathcal{H}_1(p) \leq \mathcal{H}_\beta(p)$, where $0<\beta<1<\alpha$ follows from Lemma \ref{entropyalpha1}.
\end{proof}

\begin{lemma}\label{entropymax}
$\mathcal{H}_\alpha(p) \leq \log N$ and it reaches maximum when distribution is uniform.
\end{lemma}

\begin{proof}
Let $1 \leq m \leq N$ be the number of non-zero probabilities. Then we have:
\begin{align*}
\lim_{\alpha\to 0+} \mathcal{H}_\alpha(p)= \lim_{\alpha\to 0+} \frac{1}{1-\alpha}\log\left( \sum_{k=1}^{N} p_k^{\alpha}\right) = \log m \leq \log N.
\end{align*}
So $\mathcal{H}_\alpha(p) \leq \log N$ due to Lemma \ref{entropydecr}. For the second part, put $p_1=\ldots=p_N=\frac1N$.
\begin{align*}
\mathcal{H}_1(p)=-\sum_{k=1}^{N} \frac1N\log\left(\frac1N\right)= \log N.
\end{align*}
\begin{align*}
\mathcal{H}_\alpha(p) = \frac{1}{1-\alpha}\log\left( \sum_{k=1}^{N} \frac{1}{N^{\alpha}}\right) =\frac{1}{1-\alpha}\log\left( \frac{N}{N^{\alpha}}\right) = \log N.
\end{align*}
\end{proof}

\begin{remark}\label{rem:0andinf}
Let $1\leq m \leq N$ be the number of non-zero probabilities and without loss of generality let $p_k < p_1=...=p_{N_1}$ for every $N_1+1\leq k\leq N$. Then we can also define:
\begin{align*}
\mathcal{H}_0(p) := \lim_{\alpha\to 0+} \mathcal{H}_\alpha(p)= \lim_{\alpha\to 0+} \frac{1}{1-\alpha}\log\left( \sum_{k=1}^{N} p_k^{\alpha}\right) = \log m.
\end{align*}
\begin{align*}
\mathcal{H}_\infty(p) :&= \lim_{\alpha\to +\infty} \mathcal{H}_\alpha(p)= \lim_{\alpha\to +\infty} \frac{1}{1-\alpha}\log\left( \sum_{k=1}^{N} p_k^{\alpha}\right) =
\\ &=
\lim_{\alpha\to +\infty} -\frac{\sum_{k=1}^{N} p_k^{\alpha}\log p_k}{\sum_{k=1}^{N} p_k^{\alpha}}
\\ &=
\lim_{\alpha\to +\infty} -\frac{N_1\log p_1+\sum_{k=N_1+1}^{N} \left(\frac{p_k}{p_1}\right)^{\alpha}\log p_k}{N_1+\sum_{k=N_1+1}^{N} \left(\frac{p_k}{p_1}\right)^{\alpha}}
\\ &=
-\log p_1.
\end{align*}
\end{remark}

\bigskip

\noindent Filipp Buryak\\  
filippburyak2000@gmail.com
\bigskip 

\noindent {\small
\noindent Taras Shevchenko National University of Kyiv\\
Department of Probability Theory, Statistics and Actuarial Mathematics\\
64, Volodymyrs’ka St., 01601 Kyiv, Ukraine
}\bigskip

\noindent Yuliya Mishura (corresponding author)\\
myus@univ.kiev.ua\\
ORCID: \url{https://orcid.org/0000-0002-6877-1800} \bigskip

\noindent {\small
\noindent Taras Shevchenko National University of Kyiv\\
Department of Probability Theory, Statistics and Actuarial Mathematics\\
64, Volodymyrs’ka St., 01601 Kyiv, Ukraine
}


\begin{thebibliography}{9}
\bibitem{Renyi}
A. R{\'e}nyi, \textit{On Measures of Entropy and Information}, Proceedings of the 4th Berkeley Symposium on Mathematics, Statistics and Probability; The Regents of the University of California: Berkeley, CA, USA. -- 1961. -- Vol.~1. -- P. 547--561.

\bibitem{Siu}
S.-W. Ho, S. Verd{\'u}, \textit{Convexity/concavity of R{\'e}nyi entropy and $\alpha$-mutual information}, IEEE International Symposium on Information Theory (ISIT); -- 2015. -- P. 745--749.

\bibitem{RenyiDiver}
T. van Erven, P. Harremo{\"e}s, \textit{R{\'e}nyi Divergence and Kullback-Leibler Divergence}, IEEE Transactions on Information Theory (ISIT); -- 2014. -- Vol.~60. -- P. 3797--3820.

\bibitem{CommonDiver}
M. Gil, F. Alajaji, T. Linder, \textit{R{\'e}nyi divergence measures for commonly used univariate continuous distributions}, Information Sciences; -- 2013. -- Vol.~249. -- P.~124-131.

\bibitem{DiverMajor}
T. van Erven, P. Harremo{\"e}s, \textit{R{\'e}nyi Divergence and majorization}, IEEE Transactions on Information Theory (ISIT); -- 2010. -- INSPEC~11434178.

\bibitem{MajorApp}
I. Sason, \textit{Tight Bounds on the R{\'e}nyi Entropy via Majorization with Applications to Guessing and Compression}, Entropy; -- 2018. -- Vol.~20(12). -- P.~896

\bibitem{BayesBound}
L. Bégin, P. Germain, F. Laviolette, J.-F. Roy, \textit{PAC-Bayesian Bounds based on the R{\'e}nyi Divergence}, Proceedings of the 19th International Conference on Artificial Intelligence and Statistics; -- 2016. -- PMLR 51:435-444, 2016.

\bibitem{StatMech}
E. K. Lenzi, R. S. Mendes, L. R. da Silva, \textit{Statistical mechanics based on R{\'e}nyi entropy}, Physica A: Statistical Mechanics and its Applications; -- 2000. -- Vol.~280. -- P.~337-345.

\bibitem{Compl}
J. Acharya, A. Orlitsky, A. T. Suresh, H. Tyagi, \textit{The Complexity of Estimating Rényi Entropy}, Proceedings of the 2015 Annual ACM-SIAM Symposium on Discrete Algorithms; -- 2015. -- Book Code:PRDA15.

\bibitem{Bound}
S. G. Bobkov, A. Marsiglietti, J. Melbourne, \textit{Concentration functions and entropy bounds for discrete log-concave distributions}, arXiv:2007.11030v1 [math.PR].

\bibitem{PH}
P. Harremo{\"e}s, \textit{Interpretations of R{\'e}nyi entropies and divergences}, Physica A: Statistical Mechanics and its Applications; -- 2006.

\bibitem{BaP}
P. Harremo{\"e}s, \textit{Binomial and Poisson Distributions as Maximum Entropy Distributions}, IEEE transactions on information theory; -- 2001. -- Vol.~47. -- No.~5. -- P. 2039--2041.

\end{thebibliography}
\end{document}